\documentclass[11pt,reqno]{amsart}

\usepackage[text={160mm,240mm},centering]{geometry}            
\usepackage{amssymb,amsmath,setspace,geometry,indentfirst,changepage,inputenc,amsthm}
\usepackage{mathrsfs,amsfonts,savesym,graphicx,bm,color}

\usepackage{graphicx}
\usepackage{amssymb}
\usepackage{dsfont}

\usepackage{lscape}

\usepackage{tikz}
\usetikzlibrary{positioning}

\usepackage{float}

\usepackage[all,2cell]{xy}    
\UseAllTwocells               

\usepackage{amsmath,amsthm}

\usepackage[mathscr]{eucal}
\usepackage[all]{xy}
\usepackage{mathrsfs}
\usepackage{hyperref}
\usepackage{color}

\newtheorem{theorem}{Theorem}[section]
\newtheorem{lemma}[theorem]{Lemma}

\newtheorem{conjecture}[theorem]{Conjecture}
\newtheorem{corollary}[theorem]{Corollary}
\newtheorem{proposition}[theorem]{Proposition}

\newtheorem{definition-lemma}[theorem]{Definition-Lemma}
\newtheorem{definition-theorem}[theorem]{Definition-Theorem}

\newtheorem*{maintheoremA}{Theorem A}
\newtheorem*{maintheoremB}{Theorem B}
\newtheorem*{maintheoremC}{Theorem C}

\theoremstyle{definition}

\newtheorem{definition}[theorem]{Definition}

\newtheorem{remark}[theorem]{Remark}

\usepackage{fancyhdr}
\allowdisplaybreaks

\title[The $n/d$ Root for Homogeneous Ideals]
{A Polynomial PDE Criterion for the $-n/d$ Root of the Bernstein-Sato polynomial of Homogeneous Ideals}

\author{Yifan Chen}
\address{
	Department of Mathematical Sciences,
	Tsinghua University,
	Beijing, 100084, P. R. China.}
\email{c-yf20@tsinghua.org.cn}

\author{Huaiqing Zuo}
\address{Department of Mathematical Sciences,
	Tsinghua University,
	Beijing, 100084, P. R. China.}
\email{hqzuo@mail.tsinghua.edu.cn}

\begin{document}
	
	\maketitle
	
	%

 \begin{abstract}
 	Let $I\subseteq\mathbb C[x_1,\ldots,x_n]$ be an ideal generated by homogeneous polynomials of a common degree $d$. We give a polynomial partial differential equation criterion guaranteeing that $-n/d$ is a root of the Bernstein-Sato polynomial $b_I(s)$. We apply this criterion to the ideal of maximal minors of a generic $m\times n$ matrix and obtain the distinguished root $-n$; combined with local divisibility along determinantal strata, this allows us to obtain the strong monodromy conjecture in the maximal-minor case. Finally, we prove that the criterion is stable under enlarging the linear span of the generators, adjoining generators in disjoint variables, products satisfying the natural slope condition, and Thom-Sebastiani sums. These stability results provide new classes of homogeneous ideals and polynomials for which the distinguished Bernstein-Sato root can be detected.
 	
		
 \end{abstract}

	\tableofcontents

\section{Introduction}

Let $X$ be a smooth complex algebraic variety and $f$ be a nonconstant regular function on $X$. The Bernstein-Sato polynomial $b_f(s)$ is the monic polynomial of smallest degree for which there is a differential operator $P(s) \in \mathcal{D}_{X}[s]$, satisfying $P(s)f^{s+1}=b_f(s)f^s$, where $s$ is an independent variable. The theory of the Bernstein-Sato polynomial originated in the work of Bernstein and Sato in the early 1970s. It is a basic invariant connecting $\mathcal{D}$-module theory with the geometry and topology of singularities. For example, the exponentials of the local roots recover the eigenvalues of Milnor monodromy \cite[Theorem 1]{Kashiwara_Bernstein_and_eigenvalue_of_monodromy}\cite[Theorem 3.4]{Malgrange_Bernstein_eigenvalue}. Moreover, roots in the interval $[-1,0)$ detect jumping coefficients of multiplier ideals \cite[Theorem 2]{Bernstein_Sato_polynomial_of_arbitrary_varieties}. 

A particularly important root problem arises from the monodromy conjecture. Budur, Musta\c{t}\u{a}, and Teitler formulated the $n/d$ conjecture: if a homogeneous polynomial $f$ of degree $d$ on $\mathbb C^n$ defines an essential indecomposable central hyperplane arrangement, then $b_f(-n/d)=0$ \cite[Conjecture 1.2]{BMT11}. They proved the weak monodromy conjecture for hyperplane arrangements and reduced the strong version to this statement \cite[Theorem 1.3]{BMT11}. Subsequent approaches included Walther's cohomological sufficient condition \cite[Theorem 4.12]{Wal05}, the result for generic multiplicities of Shi and Zuo \cite[Theorem 1.7]{SZ24}, and the nonresonant case of Xie and Yu \cite[Theorem 1.5]{XY25}. The conjecture has recently been proved in full by Davis and Yang using multivariate $V$-filtrations \cite[Theorems 1.2 and 1.3]{DY26}.

Nowadays, Budur, Musta\c{t}\u{a}, and Saito have extended the construction of the Bernstein-Sato polynomial from hypersurfaces to arbitrary closed subschemes. It is defined by a $V$-filtration and is independent of the chosen generators \cite[Sections 2.4 and 2.10]{Bernstein_Sato_polynomial_of_arbitrary_varieties}\cite[Theorem 2.5]{Bernstein_Sato_polynomial_of_arbitrary_varieties}. In this paper we study the ideal-theoretic $n/d$ root problem. Let $I=(f_1,\ldots,f_r)\subseteq\mathbb C[x_1,\ldots,x_n]$ be an ideal generated by homogeneous polynomials of a common degree $d$. Our first result gives a polynomial partial differential equation criterion for the distinguished equality $b_I(-n/d)=0$.

\begin{maintheoremA}[Theorem \ref{PDE_theorem}]
	Let $I=(f_1,\ldots,f_r)\subseteq\mathbb C[x_1,\ldots,x_n]$ be an ideal generated by homogeneous polynomials of a common degree $d$. Suppose that, for every $\alpha=(\alpha_1,\ldots,\alpha_r)\in\mathbb Z_{\geq 0}^r$, there is a homogeneous polynomial $P_\alpha\in\mathbb C[x_1,\ldots,x_n]$ of degree $d|\alpha|$ such that $P_0\neq 0$ and $x_iP_\alpha=\sum_{k=1}^r\partial_i f_k(\partial)P_{\alpha+e_k}$ for every $i$ and $\alpha$. Then $b_I(-n/d)=0$. Here $\partial_i f_k(\partial)$ means the differential operator obtained by replacing every $x_{i}$ in $\partial_{i}f_{k}$ with the corresponding $\partial_{i}$ and $e_{k}=(0,...,1,...,0)$ with its only non-zero term in the $k$-th position.
\end{maintheoremA}

The proof begins with a $\mathcal D$-module associated with the generators of $I$. A weighted Euler argument turns the desired root into the nonvanishing of a distinguished class in a quotient of this module. After a change of coordinates, the vanishing of that class is equivalent to the existence of a finite polynomial solution of an auxiliary PDE system. Duality with respect to the Fischer pairing converts the failure of such a finite solution into the family $\{P_\alpha\}$ in Theorem A.

Our principal determinantal application concerns maximal minors.

\begin{maintheoremB}[Theorem \ref{n/d_of_maximal_minors}]
	Let $I_m$ be the ideal of maximal minors of a generic $m\times n$ matrix, where $m\leq n$. Then the generators of $I_m$ satisfy the criterion in Theorem A; in particular, $b_{I_m}(-n)=0$.
\end{maintheoremB}

The Bernstein-Sato polynomial appearing in Theorem B was previously computed completely by L\H{o}rincz, Raicu, Walther, and Weyman. After transposing their matrix convention, \cite[Theorem 4.1]{Bernstein_for_maximal_minors} gives
$b_{I_m}(s)=\prod_{j=n-m+1}^{n}(s+j)$
for the ideal $I_m$ of maximal minors of a generic $m\times n$ matrix with $m\leq n$. In particular, $s+n$ divides $b_{I_m}(s)$, and hence $b_{I_m}(-n)=0$. Thus Theorem~B recovers, to this extent, one consequence of their stronger formula by a different method: instead of computing the entire Bernstein-Sato polynomial, our argument constructs an explicit polynomial PDE certificate for the distinguished root $-n$. This method yielding this distinguished root is enough to prove the strong monodromy conjecture for maximal minors of determinantal varieties (see Remark \ref{SMC_for_det}). 

The criterion is also stable under several natural operations, including Thom-Sebastiani products and sums. Note that the Bernstein-Sato polynomial for ideals satisfies the multiplicative Thom-Sebastiani property by \cite[Theorem 1.4]{Xu_Thom}. Write (P) for the existence of the polynomial family occurring in Theorem A.

\begin{maintheoremC}[Theorems \ref{linear_span}, \ref{ideal_sum}, \ref{Tom_product} and \ref{Tom_sum}]
	The following assertions hold.
	\begin{enumerate}
		\item If $g_1,\ldots,g_l$ and $f_1,\ldots,f_r$ are homogeneous of the same degree, each $g_a$ belongs to the linear span of $f_1,\ldots,f_r$, and $g_1,\ldots,g_l$ satisfy (P), then the $f_1,\ldots,f_r$ also satisfy (P).
		
		\item Assume $f_{1},....,f_{r} \in \mathbb{C}[x_{1},...,x_{n}]$ and $g_{1},...,g_{l} \in \mathbb{C}[y_{1},...,y_{m}]$ are homogeneous polynomials of degree $d$. If $f_{1},...,f_{r}$ and $g_{1},...,g_{l}$ both satisfy (P), then regarded as polynomials in $\mathbb{C}[x_{1},...,x_{n},y_{1},...,y_{m}]$, the polynomials $f_{1},...,f_{r},g_{1},...,g_{l}$ also satisfy (P).
		
		\item If $f\in\mathbb C[x_1,\ldots,x_n]$ and $g\in\mathbb C[y_1,\ldots,y_m]$ have degrees $d$ and $e$, respectively, each satisfies (P), and $n/d=m/e$, then $fg$ satisfies (P).
		
		\item Suppose $f \in \mathbb{C}[x_{1},...,x_{n}],g \in \mathbb{C}[y_{1},...,y_{m}]$ are homogeneous polynomials of degree $d(d \ge 2)$. If $f$ and $g$ satisfy (P) as a single polynomial, then $f+g$ also satisfies (P) as a single polynomial in $\mathbb{C}[x_{1},...,x_{n},y_{1},...,y_{m}]$.
	\end{enumerate}
\end{maintheoremC}

The paper is organized as follows. Section \ref{sec2} recalls Bernstein-Sato polynomials for functions and ideal sheaves, together with motivic zeta functions and the monodromy conjecture. Section \ref{sec3} proves the PDE criterion. Section \ref{sec4} treats isolated singularities and maximal minors and establishes the stability properties summarized in Theorem C.
	
	\section{Bernstein-Sato Polynomials and the Monodromy Conjecture}\label{sec2}
	\subsection{Bernstein-Sato Polynomials of Regular Functions}
	Throughout Section 2, let $X$ be a smooth complex algebraic variety over $\mathbb C$, let $\mathcal O_X$ be its structure sheaf, and let $\mathcal D_X$ be the sheaf of $\mathbb C$-linear differential operators on $X$. For every open subset $U\subseteq X$, write $\mathcal D_U=\mathcal D_X|_U$. All constructions below are local on $X$. Let us start with the case of a single polynomial. One can refer to \cite{Popa_ref} to get a detailed introduction.
	\begin{definition}
	Let $f\in\Gamma(X,\mathcal O_X)$ be a nonconstant regular function. There exist $P(s) \in \mathcal{D}_{X}[s]$ and $b(s) \in \mathbb{C}[s]$, such that $P(s)f^{s+1}=b(s)f^{s}$, where $s$ is an independent variable(see, for example, \cite{Popa_ref}). All these possible $b(s)$ form an ideal, and since $\mathbb{C}[s]$ is a principal ideal domain, there exists a unique monic generator of this ideal. We will denote this polynomial by $b_{f}(s)$ and it is called the Bernstein-Sato polynomial of $f$.
	\end{definition}
	
	\begin{remark}
	This is the classical definition of Bernstein and Sato; the existence in the analytic and holonomic framework is treated by Kashiwara \cite[Theorem 1]{Rationality_of_Roots_of_B-Function_Kashiwara}, and Malgrange recalls the same definition in \cite[Theorem 3.1]{Malgrange_Bernstein_eigenvalue}.
	\end{remark}

	\begin{remark}
	By definition, the Bernstein-Sato polynomial of $f$ is the minimal polynomial of the action of $s$ on the sheaf quotient $\frac{\mathcal D_X[s]f^s}{\mathcal D_X[s]f^{s+1}}$.
	\end{remark}
	
	Kashiwara proved that the roots of the Bernstein--Sato polynomial are negative rational numbers.
	
	\begin{theorem}[\cite{Rationality_of_Roots_of_B-Function_Kashiwara}]
	The roots of Bernstein-Sato polynomials are negative rational numbers.
	\end{theorem}
	
	There also exists a local version of the Bernstein-Sato polynomial, with the following definition.
	\begin{definition}
		Let $x\in X$, and let $f_x\in\mathcal O_{X,x}$ be the germ at $x$ of a nonzero regular function $f$. The local Bernstein--Sato polynomial of $f$ at $x$, denoted by $b_{f,x}(s)$, is the unique monic polynomial of smallest degree such that
		$$
		b_{f,x}(s)f_x^s
		\in
		\mathcal D_{X,x}[s]\cdot f_x^{s+1}.
		$$
		Equivalently, there exists $P_x(s)\in\mathcal D_{X,x}[s]$
		such that $P_x(s)f_x^{s+1}=b_{f,x}(s)f_x^s$.
	\end{definition}
	
	For an open subset $U\subseteq X$, let $b_{f,U}(s)$ denote the Bernstein-Sato polynomial obtained by applying Definition~2.1 to $f|_U$. In particular, $b_f(s)=b_{f,X}(s)$.
	
	\begin{theorem}[Local-global relation for one function]\label{thm:local-global-function}
		Let $f\in\Gamma(X,\mathcal O_X)$ be a nonzero regular function. For every open subset $U\subseteq X$, $b_{f,U}(s)=\operatorname{lcm}_{x\in U}b_{f,x}(s)$.
		Consequently, if $V\subseteq U$ is open, then $b_{f,V}(s)\mid b_{f,U}(s)$
		and for every $x\in U$, $b_{f,x}(s)\mid b_{f,U}(s)$.
		In particular, $b_f(s)=\operatorname{lcm}_{x\in X}b_{f,x}(s)$,
		and hence $b_{f,x}(s)\mid b_f(s)$
		for every $x\in X$.
	\end{theorem}

	There is another equivalent definition of Bernstein-Sato polynomial by V-filtrations. The V-filtration is originally due to Kashiwara and Malgrange; see Kashiwara \cite[Theorem 1]{Kashiwara_Bernstein_and_eigenvalue_of_monodromy} and Malgrange \cite[Section 3, Lemma 3.4 and Theorem 3.4]{Malgrange_Bernstein_eigenvalue}. A modern rationally indexed convention compatible with Hodge module theory is used in Saito \cite[Section 3.1]{Sai88}. To introduce the V-filtration, we need some notations. Set $Y=X\times\mathbb A^1_{\mathbb C}$, and denote the coordinate on the second factor by $t$. Let $\mathcal D_Y$ be the sheaf of differential operators on $Y$, and let $\mathcal I=(t)\subseteq\mathcal O_Y$ be the ideal sheaf of $X\times\{0\}$.  We define a filtration $V^{\bullet}\mathcal{D}_{x,t}$ as 
	$$
	V^m\mathcal D_Y
	=
	\left\{
	P\in\mathcal D_Y
	\ \middle|\
	P\mathcal I^q\subseteq\mathcal I^{q+m}
	\text{ for every }q\in\mathbb Z
	\right\},
	$$
	where $\mathcal I^q=\mathcal O_Y$ for $q<0$.

	We define an operator $s:=-\partial_{t}t$ in $\mathcal{D}_{Y}$. This $s$ will be identified with the one in the definition of Bernstein-Sato polynomial. Now we give the definition of V-filtration.
	\begin{definition}
	Let $\mathcal{M}$ be a coherent $\mathcal{D}_{Y}$ module. A V-filtration along $X \times \{0\} \subset X \times \mathbb{C}$ on $\mathcal{M}$ is a decreasing filtration $V^{\bullet}\mathcal{M}$ indexed by rational numbers satisfying the following conditions.
	
	$(1)$ It is exhaustive, i.e. $\bigcup_{\alpha \in \mathbb{Q}}V^{\alpha}\mathcal{M}=\mathcal{M}$,
	
	$(2)$ It is discrete and left continuous. This means that there exists a positive integer $N$ large enough such that for every $i \in \mathbb{Z}$ and $\alpha \in (\frac{i}{N},\frac{i+1}{N})$, $V^{\alpha}\mathcal{M}$ is constant and $V^\alpha M=\bigcap_{\beta<\alpha}V^\beta M$ for every $\alpha\in\mathbb Q$.
	
	$(3)$ For every $\alpha \in \mathbb{Q}$ and $i \in \mathbb{Z}$, we have $V^{i}\mathcal{D}_{Y} \cdot V^{\alpha}\mathcal{M} \subset V^{i+\alpha}\mathcal{M}$.
	
	$(4)$ For every $\alpha \in \mathbb{Q}$, the operator $s+\alpha$ is nilpotent on $\mathrm{Gr}_{V}^{\alpha}\mathcal{M}:=V^{\alpha}\mathcal{M}/\bigcup_{\beta>\alpha}V^{\beta}\mathcal{M}$.
	
	$(5)$ For every $\alpha \in \mathbb{Q}$, $V^{\alpha}\mathcal{M}$ is finitely generated over $V^{0}\mathcal{D}_{Y}$.
	
	$(6)$ We have $t \cdot V^{\alpha}\mathcal{M}=V^{\alpha+1}\mathcal{M}$ for $\alpha>0$.
	\end{definition}
	
	The existence theorem and uniqueness theorem for V-filtration due to Kashiwara and Malgrange are as follows. 
	\begin{theorem}[Existence theorem]\cite[Theorem 1]{Kashiwara_Bernstein_and_eigenvalue_of_monodromy}, \cite[Lemma 3.3 and Theorem 3.4]{Malgrange_Bernstein_eigenvalue}
		Let $\mathcal{M}$ be a regular holonomic $\mathcal{D}_{Y}$-module with quasi-unipotent local monodromy along $\{t=0\}$. Then $\mathcal{M}$ admits a $V$-filtration.
	\end{theorem} 

	\begin{theorem}[Uniqueness theorem]\cite[Theorem 1]{Kashiwara_Bernstein_and_eigenvalue_of_monodromy}
		If a coherent $\mathcal{D}_Y$-module $\mathcal{M}$ admits a Kashiwara--Malgrange $V$-filtration, then this filtration is unique.
		
		More explicitly, if $U^{\alpha}\mathcal{M}$ and $V^{\alpha}\mathcal{M}$ are two decreasing rational filtrations satisfying the defining properties above, then $U^{\alpha}\mathcal{M}=V^{\alpha}\mathcal{M}$
		for every $\alpha\in \mathbb{Q}$.
	\end{theorem}
	
	Now we will use the V-filtration to give another equivalent definition of the Bernstein-Sato polynomial. Let $f\in\Gamma(X,\mathcal O_X)$ be a nonzero regular function and let $i_f:X \hookrightarrow X \times \mathbb{A}^{1}_{\mathbb{C}}$
	be the graph embedding, with coordinate $t$ on the second factor. We consider the module $(i_f)_+\mathcal{O}_X$. One can define the $V$-filtration of the direct image$(i_f)_+\mathcal{O}_{X}$ since it is a regular holonomic $\mathcal{D}_{Y}$-module.

	Direct computation shows that locally on $X$, the $\mathcal D_Y$-module $(i_f)_+\mathcal O_X$ is identified with the formal module $\mathcal D_Yf^s$. 
	Consider the ambient module $\mathcal N_f=\mathcal O_X[f^{-1},s]f^s$,
	with actions $t(h(x,s)f^s)=f\,h(x,s+1)f^s$, $\partial_t(h(x,s)f^s)=-s f^{-1}h(x,s-1)f^s$.
	The graph module $(i_f)_+\mathcal O_X$ is identified with
	the cyclic submodule $\mathcal D_Yf^s\subseteq\mathcal N_f$.
	Expressions involving $f^{s-k}$ or $t^{-1}$ are understood
	in the ambient module. This module $(i_f)_+\mathcal O_X$ has the following characterization.
	
	\begin{proposition}\cite{Mal75}
	We have $$
	(i_f)_+\mathcal O_X
	\simeq
	\frac{\mathcal D_Y}
	{\mathcal D_Y\left(
		t-f,\ 
		\xi+\xi(f)\partial_t
		\ \middle|\
		\xi\in\Theta_X
		\right)}.
	$$
	In local coordinates $x_1,\ldots,x_n$ on $X$, this becomes
	$$
	(i_f)_+\mathcal O_X
	\simeq
	\frac{\mathcal D_Y}
	{\mathcal D_Y\left(
		t-f,\ 
		\partial_{x_i}+\partial_{x_i}(f)\partial_t
		\ \middle|\
		1\leq i\leq n
		\right)}.
	$$
	\end{proposition}

	\begin{definition}
	For any $\delta \in (\iota_f)_+\mathcal{O}_X$, we define the Bernstein-Sato polynomial of $\delta$, denoted by $b_{\delta}(s)$, to be the minimal polynomial of the action of $s=-\partial_{t}t$ on $\frac{V^{0}\mathcal{D}_{Y} \cdot \delta}{V^{1}\mathcal{D}_{Y} \cdot \delta}$.
	\end{definition}
	
	One can verify the following proposition by calculating the quotient $\frac{V^{0}\mathcal{D}_{Y} \cdot \delta}{V^{1}\mathcal{D}_{Y} \cdot \delta}$ explicitly.
	\begin{proposition}
	With the notation above, if we take $\delta=1$, then $b_{\delta}(s)=b_{f}(s)$.
	\end{proposition}
	
	\subsection{Bernstein-Sato Polynomials of Ideal Sheaves}
	Budur, Mustaţă, and Saito introduced the Bernstein-Sato polynomial for arbitrary closed subschemes, or equivalently for coherent ideals, by using the $V$-filtration, see \cite[Sections 2.4 and 2.10]{Bernstein_Sato_polynomial_of_arbitrary_varieties}. This construction extends the classical hypersurface case and is independent of the chosen generators of the ideal by \cite[Theorem 2.5]{Bernstein_Sato_polynomial_of_arbitrary_varieties}.
	
	Before we give the definition of Bernstein-Sato polynomial for an ideal, let us fix some notations. Let $\mathfrak a\subseteq\mathcal O_X$ be a nonzero coherent ideal sheaf. Let $U\subseteq X$ be an open subset on which $\mathfrak a$ is generated by nonzero regular functions
	$\mathfrak a|_U=(f_1,\ldots,f_r)$. Let $s_1,\ldots,s_r$ be independent variables, and put $s=s_1+\cdots+s_r$.
	We consider the formal symbol $f_1^{s_1}\cdots f_r^{s_r}$
	inside $\mathcal O_U[f_1^{-1},\ldots,f_r^{-1},s_1,\ldots,s_r]
	f_1^{s_1}\cdots f_r^{s_r}$. For a positive integer $m$, we use the binomial notation $\binom{s_i}{m}=\frac{s_i(s_i-1)\cdots(s_i-m+1)}{m!}$.
	
	\begin{definition}[Bernstein-Sato polynomial of an ideal]\label{def_b_function_ideal_open_subset}
		With the notation above, the Bernstein-Sato polynomial of $\mathfrak a$ on $U$, denoted by $b_{\mathfrak a,U}(s)$, is the monic polynomial of smallest degree such that
		\begin{align}\label{def_b_function_ideal}
			b_{\mathfrak{a},U}(s)f_1^{s_1}\cdots f_r^{s_r}
			\in
			\sum_{\substack{u\in \mathbb{Z}^r\\ |u|=1}}
			\mathcal{D}_U[s_1,\ldots,s_r]\cdot
			\left(
			\prod_{u_i<0}\binom{s_i}{-u_i}
			\right)
			f_1^{s_1+u_1}\cdots f_r^{s_r+u_r},
		\end{align}
		where $|u|=u_1+\cdots+u_r$.
		This is the formulation used by Mustaţă in \cite[Introduction]{Mus22}; it is equivalent to the definition of Budur, Mustaţă, and Saito in \cite[Section 2.10]{Bernstein_Sato_polynomial_of_arbitrary_varieties}.
	\end{definition}
	
	In the case $r=1$, the only vector $u\in\mathbb Z$ with $|u|=1$ is $u=1$, and the definition reduces to $b_{\mathfrak a,U}(s)f^s
	\in
	\mathcal D_U[s]\cdot f^{s+1}$.
	Thus, if $\mathfrak a|_U=(f)$, then $b_{\mathfrak a,U}(s)=b_{f,U}(s)$, where $b_{f,U}(s)$ denotes the Bernstein-Sato polynomial of the restriction of $f$ to $U$. Hence $b_{\mathfrak a,U}(s)$ extends the Bernstein-Sato polynomial of a single regular function to coherent ideal sheaves.

	The binomials may be hard to understand at the first sight of the expression. We can explain them through the view of V-filtrations. Let
	$\mathbf f=(f_1,\ldots,f_r):U\longrightarrow\mathbb A^r$
	and let $i_{\mathbf f}:U\hookrightarrow Y:=U\times\mathbb A^r$
	be its graph embedding. Denote the coordinates on $\mathbb{A}^r$ by $t_1,\ldots,t_r$.
	The direct image $(i_{\mathbf{f}})_+\mathcal{O}_U$
	is a regular holonomic $\mathcal{D}_{Y}$-module. As in the case of a single polynomial, we have $(i_{\mathbf f})_+\mathcal{O}_U \cong \mathcal{D}_{Y}f_{1}^{s_{1}}...f_{r}^{s_{r}}$, where $\mathcal D_Y$ denotes the sheaf of differential operators on $Y$. 
	Consider the ambient module $\mathcal N_{\mathbf{f}}=\mathcal O_U[(f_{1}...f_{r})^{-1},s_{1},...,s_{r}]f_{1}^{s_{1}}...f_{r}^{s_{r}}$,
	with actions
	$$t_{i}(h(x,s_{1},...,s_{r})f_{1}^{s_{1}}...f_{r}^{s_{r}})=f_{i}\,h(x,s_{1},...,s_{i}+1,...,s_{r})f_{1}^{s_{1}}...f_{r}^{s_{r}}$$ and
	
	$$\partial_{t_{i}}(h(x,s_{1},...,s_{r})f_{1}^{s_{1}}...f_{r}^{s_{r}})=-s_{i}f_{i}^{-1}\,h(x,s_{1},...,s_{i}-1,...,s_{r})f_{1}^{s_{1}}...f_{r}^{s_{r}}$$.
	The graph module $(i_\mathbf{f})_+\mathcal O_U$ is identified with
	the cyclic submodule $\mathcal D_Yf^s\subseteq\mathcal N_f$.
	Expressions involving $t_{i}^{-1}$ are understood
	in the ambient module. Let $s_{ij}:=-\partial_{t_{i}}t_{j}$. One can verify that the definition for $b_{\mathfrak{a},U}(s)$ is equivalent to the minimal polynomial of the action of $s=\sum_{i=1}^{r}-\partial_{t_{i}}t_{i}$ on 
	$$\frac{\mathcal{D}_{U}[s_{i},s_{ij}]f_{1}^{s_{1}}...f_{r}^{s_{r}}}{\sum_{k=1}^{r}\mathcal{D}_{U}[s_{i},s_{ij}]f_{k}\cdot f_{1}^{s_{1}}...f_{r}^{s_{r}}}.$$
	
	For the module $\mathcal{D}_{x,t}f_{1}^{s_{1}}...f_{r}^{s_{r}}$, we also have the following characterization. The proof is the same as in the case of a single polynomial.
	\begin{proposition}
$$
(i_{\mathbf f})_+\mathcal O_U
\simeq
\frac{\mathcal D_Y}
{\mathcal D_Y\left(
	t_k-f_k,\ 
	\xi+\sum_{k=1}^r\xi(f_k)\partial_{t_k}
	\ \middle|\
	1\leq k\leq r,\ \xi\in\Theta_U
	\right)}.
$$
In local coordinates $x_1,\ldots,x_n$ on $U$, the second family of generators is
$$
\partial_{x_i}
+
\sum_{k=1}^r
\partial_{x_i}(f_k)\partial_{t_k},
\qquad
1\leq i\leq n.
$$
	\end{proposition}
	
	Since $(i_\mathbf{f})_+\mathcal{O}_U$
	is a regular holonomic $\mathcal{D}_{Y}$-module, one can define the V-filtration along the smooth subvariety $U\times\{0\}\subset Y$. More explicitly, let $\mathcal J=(t_1,\ldots,t_r)\subseteq\mathcal O_Y$, and we define the filtration $V^{\bullet}\mathcal{D}_{Y}$ as
	\begin{flalign*}
	V^m\mathcal D_Y
	=
	\left\{
	P\in\mathcal D_Y
	\ \middle|\
	P\mathcal J^q\subseteq\mathcal J^{q+m}
	\text{ for every }q\in\mathbb Z
	\right\},
	\end{flalign*}
	where $\mathcal J^q=\mathcal O_Y$ for $q<0$.

	\begin{definition}
		Let $M$ be a coherent $\mathcal D_Y$-module. A V-filtration along $U \times \{0\} \subset Y$ on $\mathcal{M}$ is a decreasing filtration $V^{\bullet}\mathcal{M}$ indexed by rational numbers satisfying the following conditions.
		
		$(1)$ It is exhaustive, i.e. $\bigcup_{\alpha \in \mathbb{Q}}V^{\alpha}\mathcal{M}=\mathcal{M}$,
		
		$(2)$ It is discrete and left continuous. This means that there exists a positive integer $N$ large enough such that for every $i \in \mathbb{Z}$ and $\alpha \in (\frac{i}{N},\frac{i+1}{N})$, $V^{\alpha}\mathcal{M}$ is constant and $V^\alpha M=\bigcap_{\beta<\alpha}V^\beta M$ for every $\alpha\in\mathbb Q$.
		
		$(3)$ For every $\alpha \in \mathbb{Q}$ and $i \in \mathbb{Z}$, we have $V^{i}\mathcal{D}_{x,t} \cdot V^{\alpha}\mathcal{M} \subset V^{i+\alpha}\mathcal{M}$.
		
		$(4)$ For every $\alpha \in \mathbb{Q}$, the operator $s+\alpha:=-\sum_{i=1}^{r}\partial_{t_{i}}t_{i}+\alpha$ is nilpotent on $\mathrm{Gr}_{V}^{\alpha}\mathcal{M}:=V^{\alpha}\mathcal{M}/\bigcup_{\beta>\alpha}V^{\beta}\mathcal{M}$.
		
		$(5)$ For every $\alpha \in \mathbb{Q}$, $V^{\alpha}\mathcal{M}$ is finitely generated over $V^{0}\mathcal{D}_{Y}$.
		
		$(6)$ For every $i>0$ and every sufficiently large $\alpha$,
		$V^i\mathcal D_Y\cdot V^\alpha M
		=
		V^{\alpha+i}M$.
	\end{definition}
	
	The following is the definition of Bernstein-Sato polynomial for ideals, and one can verify that this is equivalent to the equation \ref{def_b_function_ideal}. This equivalence is also explained in \cite[Sections 2.4 and 2.10]{Bernstein_Sato_polynomial_of_arbitrary_varieties}.
	\begin{definition}
		For any $\delta \in (i_f)_+\mathcal{O}_U$, we define the Bernstein-Sato polynomial of $\delta$, denoted by $b_{\delta}(s)$, to be the minimal polynomial of the action of $s=\sum_{k=1}^{r}-\partial_{t_{k}}t_{k}$ on $\frac{V^{0}\mathcal{D}_{Y} \cdot \delta}{V^{1}\mathcal{D}_{Y} \cdot \delta}$.
	\end{definition}
	
	\begin{proposition}
		With the notation above, if we take $\delta=1$, then $b_{\delta}(s)=b_{\mathfrak{a},U}(s)$.
	\end{proposition}
	
	The existence and the independence of generators are proved by Budur, Mustaţă, and Saito in \cite[Theorem 2.5]{Bernstein_Sato_polynomial_of_arbitrary_varieties}. 
	\begin{theorem}[Existence and independence of generators]
	Let $\mathfrak a\subseteq\mathcal O_X$ be a nonzero coherent ideal sheaf. On an open subset $U\subseteq X$, suppose that
	$\mathfrak a|_U=(f_1,\ldots,f_r)=(g_1,\ldots,g_m)$.
	Then $b_{\mathfrak a,U}(s)$ exists, and the two systems of generators define the same polynomial.
	\end{theorem}
	
	As in the case of a single polynomial, there is also a version of the local Bernstein-Sato polynomial for ideals. 	
	\begin{definition}
		Let $\mathfrak a\subseteq\mathcal O_X$ be a nonzero coherent ideal sheaf and let $x\in X$. Choose generators $\mathfrak a_x=(f_1,\ldots,f_r)$ in $\mathcal O_{X,x}$, and put $s=s_1+\cdots+s_r$.
		The local Bernstein-Sato polynomial of $\mathfrak a$ at $x$, denoted by $b_{\mathfrak a,x}(s)$, is the unique monic polynomial of smallest degree such that $b_{\mathfrak a,x}(s_1+\cdots+s_r)
		f_1^{s_1}\cdots f_r^{s_r}$
		belongs to
		$
		\sum_{\substack{u\in\mathbb Z^r\\ |u|=1}}
		\mathcal D_{X,x}[s_1,\ldots,s_r]
		\left(
		\prod_{u_i<0}\binom{s_i}{-u_i}
		\right)
		f_1^{s_1+u_1}\cdots f_r^{s_r+u_r}$.
		By \cite[Theorem 2.5]{Bernstein_Sato_polynomial_of_arbitrary_varieties}, this polynomial is independent of the chosen generators of $\mathfrak a_x$.
	\end{definition}
	
	\begin{definition}
		For every open subset $U\subseteq X$, define
		$$
		b_{\mathfrak a,U}(s)
		:=
		\operatorname{lcm}_{x\in U}b_{\mathfrak a,x}(s).
		$$
		If $\mathfrak a|_U$ is generated by regular functions on $U$, this definition agrees with Definition \ref{def_b_function_ideal_open_subset} by \cite[Section 2.10, formula (2.10.1)]{Bernstein_Sato_polynomial_of_arbitrary_varieties}. The global Bernstein-Sato polynomial of $\mathfrak a$ is defined by
		$$
		b_{\mathfrak a}(s)
		:=
		b_{\mathfrak a,X}(s)
		=
		\operatorname{lcm}_{x\in X}b_{\mathfrak a,x}(s).
		$$
	\end{definition}
	
	\begin{theorem}[Local-global relation for ideals]\label{thm:local-global-ideal}
		Let $\mathfrak a\subseteq\mathcal O_X$ be a nonzero coherent ideal sheaf. If $V\subseteq U\subseteq X$ are open subsets, then $b_{\mathfrak a|_V,V}(s)
		\mid
		b_{\mathfrak a|_U,U}(s)$.
		Moreover, for every $x\in U$,
		$b_{\mathfrak a,x}(s)
		\mid
		b_{\mathfrak a|_U,U}(s)$.
		In particular, $b_{\mathfrak a,x}(s)\mid b_{\mathfrak a}(s)$
		for every $x\in X$.
	\end{theorem}
	
	Let $Z\subseteq X$ be the closed subscheme defined by a nonzero coherent ideal sheaf
	$\mathfrak a\subseteq\mathcal O_X$,
	and set $c=\operatorname{codim}_X Z$.
	Following Budur, Mustaţă, and Saito, define $b_Z(s):=b_{\mathfrak a}(s-c)$. This $b_{Z}$ will only depend on $Z$, and will not depend on the embedded space $X$. Combining with the theorem above, we have the following proposition.
	\begin{proposition}\label{open_subset_b_function}
	Under the assumption above, let $U \subset X$ be a non-empty open subset, Assume that $U\cap Z\ne\varnothing$ and that $\operatorname{codim}_U(U\cap Z)=\operatorname{codim}_X Z$.
	Then $b_{U\cap Z}(s)$ divides $b_Z(s)$.
	In particular, this holds if $Z$ is equidimensional.
	\end{proposition}
	Another important relation between Bernstein-Sato polynomials of varieties is that adjoining smooth affine coordinates does not change the Bernstein-Sato polynomial.
	\begin{proposition}\label{multiplying_smooth_variety}
	Under the assumption above, let $Y$ be a smooth affine variety, then $b_{Z}=b_{Z \times Y}$.
	\end{proposition}
\begin{proof}
	It is enough to compare the local Bernstein--Sato polynomials. Fix $(x,y)\in X\times Y$. Every local functional equation for $\mathfrak a$ at $x$ pulls back to a local functional equation for $\mathfrak a\mathcal O_{X\times Y}$ at $(x,y)$, and hence
	$b_{\mathfrak a\mathcal O_{X\times Y},(x,y)}(s)\mid b_{\mathfrak a,x}(s)$.
	
	Conversely, choose local coordinates on the smooth germ $(Y,y)$ and write every differential operator occurring in a local functional equation at $(x,y)$ in normal form, with all derivatives in the $Y$-directions placed on the right. Since the generators of $\mathfrak a\mathcal O_{X\times Y}$ are pulled back from $X$, every term containing a positive-order derivative in a $Y$-direction annihilates the corresponding formal symbol. Evaluating the remaining coefficient functions at $y$ therefore gives a local functional equation at $x$. Thus
	$b_{\mathfrak a,x}(s)\mid b_{\mathfrak a\mathcal O_{X\times Y},(x,y)}(s)$.
	
	Consequently,
	$b_{\mathfrak a\mathcal O_{X\times Y},(x,y)}(s)=b_{\mathfrak a,x}(s)$.
	Taking least common multiples over all points and observing that
	$\operatorname{codim}_{X\times Y}(Z\times Y)=\operatorname{codim}_X Z$
	gives $b_{Z\times Y}(s)=b_Z(s)$.
\end{proof}
	\subsection{Motivic Zeta Functions and the Monodromy Conjecture}\label{monodromy_conj}
	
	The monodromy conjecture, which was originally formulated for a single regular function, establishes the relation between the poles of the motivic zeta function, the eigenvalues of local Milnor monodromy action and the roots of Bernstein-Sato polynomials. We have introduced the Bernstein-Sato polynomial in the preceding subsection, so we will give more details about eigenvalues of the monodromy action and the zeta function.
	
	Let $f\in\Gamma(X,\mathcal O_X)$ be a nonconstant regular function. For every $x\in f^{-1}(0)$, the analytification of $f$ defines a local Milnor fibre at $x$ and hence a monodromy action on its cohomology. We denote by $E(f)$ the set of all eigenvalues of the cohomology groups of the Milnor fiber at any point $x \in V(f)$. These eigenvalues have deep connection with the roots of the Bernstein-Sato polynomial of $f$. In fact, the authors of \cite{Kashiwara_Bernstein_and_eigenvalue_of_monodromy} and \cite{Malgrange_Bernstein_eigenvalue} showed that the set $$\mathrm{Exp}(V(b_{f})):=\{e^{2\pi \sqrt{-1}s_{0}}\; | \; b_{f}(s_{0})=0\}$$ equals $E(f)$.
	
	The other side of the monodromy conjecture concerns the motivic zeta function. It is defined via motivic integration, but it can be computed using a log resolution as in the following theorem. The original form is about a single polynomial if we take the ideal $I=(f)$, but the theorem holds for the ideal case.
	
	\begin{theorem}[\cite{Motivic_Igusa_Zeta_Function}]\label{theorem_denef_loeser}
		Let $X$ be a smooth complex algebraic variety, let $\mathfrak a\subseteq\mathcal O_X$ be a nonzero coherent ideal sheaf, and let
		$\mu:Y\longrightarrow X$ be a log resolution of $\mathfrak a$. Suppose $E_{i}(i \in S)$ are the irreducible components of $\mu^{*}(\mathfrak{a})$ and relative canonical divisor $K_{Y/X}$ such that $\mu^{*}(\mathfrak{a})=\sum_{i \in S}N_{i}E_{i}$ and $K_{Y/X}=\sum_{i \in S}(\nu_{i}-1)E_{i}$. Then the motivic zeta function of $\mathfrak{a}$, denoted by $Z^{\mathrm{mot}}_{\mathfrak{a}}(s)$, can be calculated by
		\begin{flalign*}
			Z^{\mathrm{mot}}_{\mathfrak{a}}(s)=\sum_{J \subset S}[\mathring{E}_{J}]\prod_{i \in J}\frac{\mathbb{L}-1}{\mathbb{L}^{N_{i}s+\nu_{i}}-1},
		\end{flalign*}
		where $\mathring{E}_{J}:=(\cap_{i \in J}E_{i})\setminus (\cup_{j \notin J}E_{j})$. Here we call $(N_{i},\nu_{i})(i \in S)$ the data of the log resolution.
	\end{theorem} 
	
	The motivic zeta function can be specialized into the topological zeta function by imposing the Euler characteristic $\chi(\cdot)$ into the expression, as in the following definition.
	\begin{definition}[Topological zeta function]
		As in the setting above, the topological zeta function of $\mathfrak{a}$ is given by
		\begin{flalign*}
			Z^{\mathrm{top}}_{\mathfrak{a}}(s):=\sum_{J \subset S}\chi(\mathring{E}_{J})\prod_{i \in J}\frac{1}{N_{i}s+\nu_{i}}.
		\end{flalign*}
	\end{definition}
	
	Now we can state the monodromy conjecture.
	
	\begin{conjecture}[Monodromy Conjecture, \cite{Motivic_Igusa_Zeta_Function}]\label{mc}
		Let $X$ be a smooth complex algebraic variety and let $f\in\Gamma(X,\mathcal O_X)$ be a nonconstant regular function. If $s_0$ is a pole of $Z_{f}^{\mathrm{mot}}(s)$, then $e^{s_0\cdot 2\pi \sqrt{-1}} \in E(f)$.
	\end{conjecture}

	The monodromy conjecture can be generalized to the ideal case. When we replace $f$ with a non-zero coherent ideal sheaf $\mathfrak a\subseteq\mathcal O_X$, we can replace $Z^{\mathrm{mot}}_{f}$ with $Z^{\mathrm{mot}}_{\mathfrak{a}}$. By \cite{Bernstein_Sato_polynomial_of_arbitrary_varieties}, the Bernstein-Sato polynomial can also be generalized to the ideal case, denoted by $b_{\mathfrak{a}}(s)$. Unfortunately, there is no counterpart of the local Milnor monodromy eigenvalue for the ideal version. However, Verdier introduced the notion of Verdier monodromy in \cite{Verdier_new_monodromy} which is equivalent to the Milnor monodromy version and it can be naturally generalized to the ideal case. This generalization is compatible with that of the Bernstein-Sato polynomial, since in \cite{Bernstein_Sato_polynomial_of_arbitrary_varieties} Budur, Musta\c{t}ă and Saito proved that the set $$\mathrm{Exp}(V(b_{\mathfrak{a}}(s))):=\{e^{2\pi \sqrt{-1}s_{0}}\; | \; b_{\mathfrak{a}}(s_{0})=0\}$$ equals the set of all eigenvalues of Verdier monodromy.

	\section{A Polynomial PDE Criterion for the $n/d$ Root}\label{sec3}
	From this point onward, we specialize to $X=\mathbb A_{\mathbb C}^n$.
	Thus
	$\Gamma(X,\mathcal O_X)=\mathbb C[x_1,\ldots,x_n],
	\Gamma(X,\mathcal D_X)=D_x$,
	where $D_x$ is the $n$-th Weyl algebra. All ideals considered below are generated by homogeneous polynomials of the same degree.
	\begin{definition}
	Let $I \subset \mathbb{C}[x_{1},...,x_{n}]$ be an ideal generated by homogeneous polynomials of degree $d$. We say $I$ has $\frac{n}{d}$ property if $b_{I}(-\frac{n}{d})=0$.
	\end{definition}

	Assume $X=\mathbb{C}^{n}$. Let us fix an ideal $I \subset \mathbb{C}[x_{1},...,x_{n}]$ with degree $d$ generators $f_{1},...,f_{r}$. We consider the module $\mathcal{D}_{x,t}f_{1}^{s_{1}}...f_{r}^{s_{r}}$ which is isomorphic to $\iota_{+}\mathcal{O}_{X}$. We have the following characterization of $\mathcal{D}_{x,t}f_{1}^{s_{1}}...f_{r}^{s_{r}}$.
	
	\begin{proposition}
	$\mathcal{D}_{x,t}f_{1}^{s_{1}}...f_{r}^{s_{r}} \cong \mathcal{D}_{x,t}/\mathcal{D}_{x,t}(f_{i}-t_{i},\partial_{i}+\sum_{k=1}^{r}\partial_{i}f_{k}\partial_{t_{k}})$.
	\end{proposition}
	
	Now we consider the module $\mathcal{M}:=\mathcal{D}_{x,t}f_{1}^{s_{1}}...f_{r}^{s_{r}}/(\partial_{1},...,\partial_{n})\mathcal{D}_{x,t}f_{1}^{s_{1}}...f_{r}^{s_{r}}$. We have the following proposition.
	
	\begin{proposition}
	If $1 \in \mathcal{M}$ is non-zero, then $b_{I}(-\frac{n}{d})=0$.
	\end{proposition}
	\begin{proof}
	According to the definition of Bernstein-Sato polynomial, there exist $P_{k} \in \mathcal{D}_{x}[s_{i},s_{ij}]$ such that $b_{I}(s_{1}+...+s_{r})f_{1}^{s_{1}}...f_{r}^{s_{r}}=\sum_{k=1}^{r}P_{k}f_{k} \cdot f_{1}^{s_{1}}...f_{r}^{s_{r}}$, where $s_{ij}=\partial_{t_{i}}t_{j}$. If we set deg $x_{i}=1$, deg $\partial_{i}=-1$, deg $t_{i}=d$ and deg $\partial_{t_{i}}=-d$, then $(f_{i}-t_{i},\partial_{i}+\sum_{k=1}^{r}\partial_{i}f_{k}\partial_{t_{k}})$ is a homogeneous ideal and deg $s_{ij}=s_{i}=0$. Since $\sum_{k=1}^{r}P_{k}f_{k}-b_{I}(s) \in (f_{i}-t_{i},\partial_{i}+\sum_{k=1}^{r}\partial_{i}f_{k}\partial_{t_{k}})$, we may assume deg $P_{k}=-d$. This implies that $P_{k} \in (\partial_{1},...,\partial_{n})\mathcal{D}_{x,t}$, so $0 \equiv \sum_{k=1}^{r}P_{k}f_{k}-b_{I}(s) \equiv -b_{I}(s)$ in $\mathcal{M}$. Note that since $f_{1},...,f_{r}$ are of degree $d$, we have $$\sum_{i=1}^{n}x_{i}\partial_{i} \cdot f_{1}^{s_{1}}...f_{r}^{s_{r}}=d(s_{1}+...+s_{r})f_{1}^{s_{1}}...f_{r}^{s_{r}}.$$
	This tells us $ds-\sum_{i=1}^{n}x_{i}\partial_{i} \equiv 0$ in $\mathcal{M}$, so we have $$0 \equiv b_{I}(s) \equiv b_{I}(\frac{\sum_{i=1}^{n}x_{i}\partial_{i}}{d}) \equiv b_{I}(\frac{-n+\sum_{i=1}^{n}\partial_{i}x_{i}}{d}) \equiv b_{I}(-\frac{n}{d}).$$
	This implies that non-zero $1 \in \mathcal{M}$ will force $b_{I}(-\frac{n}{d})=0$. 
	\end{proof}
	
	To characterize $\mathcal{M}$ more clearly, we consider the following change of variables.
	\begin{flalign*}
		x_{i}'=x_{i}(1 \le i \le n),t_{j}'=t_{j}-f_{j}(1 \le j \le r).
	\end{flalign*}
	Calculation shows that the corresponding differential operators will be changed into the following forms.
	\begin{flalign*}
		\partial_{i}=\partial_{i}'-\sum_{k=1}^{r}\partial_{i}f_{k}\partial_{t_{k}}'(1 \le i \le n), \partial_{t_{j}}=\partial_{t_{j}}'(1 \le j \le r).
	\end{flalign*}
	After this change of variables, we see that $\mathcal{M} \cong \mathcal{M}'=\mathcal{D}_{x,t}/\mathcal{D}_{x,t}(t_{j},\partial_{i})+(\partial_{i}-\sum_{k=1}^{r}\partial_{i}f_{k}\partial_{t_{k}})\mathcal{D}_{x,t}$. With the reduction of the ideal $(t_{j},\partial_{i})$, we may assume that every element in $\mathcal{M}'$ is of the form $\sum g_{\alpha_{1},...,\alpha_{r}}\partial_{1}^{\alpha_{1}}...\partial_{r}^{\alpha_{r}}$. This ideal is so special that we have the following lemma.
	\begin{lemma}\label{maximal_ideal}
	Let $A=\sum g_{\alpha_{1},...,\alpha_{r}}\partial_{t_{1}}^{\alpha_{1}}...\partial_{t_{r}}^{\alpha_{r}}$ be an element in $\mathcal{D}_{x,t}$ with $g_{\alpha_{1},...,\alpha_{r}} \in \mathbb{C}[x_{1},...,x_{n}]$. If $A \in \mathcal{D}_{x,t}(t_{j},\partial_{i})$, then $g_{\alpha_{1},...,\alpha_{r}}=0$ for all $\alpha_{1},...,\alpha_{r}$.
	\end{lemma}	
	\begin{proof}
	We assume that $A \neq 0$. If $A \in (t_{j},\partial_{i})$, then the commutator $[t_{k},A] \in (t_{j},\partial_{i})$ for all $k$. This commutator will give a new element with lower degree in $\partial_{t_{k}}$. We can repeat this procedure until the degree of all $\partial_{t_{j}}$ is $0$. Now we get some non-zero $g \in \mathbb{C}[x_{1},...,x_{n}]$ with $g \in (t_{j},\partial_{i})$. Similarly we have the commutator $[\partial_{i},g] \in (t_{j}, \partial_{i})$. After repeating the procedure, we will get some non-zero constant in $(t_{j}, \partial_{i})$, a contradiction, so the lemma holds.
	\end{proof}
	
	We now characterize the condition that the class of $1$ vanishes in $M'$. The equality $1=0$ in $M'$ is equivalent to the existence of $D_{1},...,D_{n}$ such that $1+\sum_{i=1}^{n}(\partial_{i}-\sum_{k=1}^{r}\partial_{i}f_{k}\partial_{t_{k}})D_{i} \in \mathcal{D}_{x,t}(t_{j},\partial_{i})$, where $D_{i}=\sum g_{i,\alpha_{1},...,\alpha_{r}}\partial_{t_{1}}^{\alpha_{1}}...\partial_{t_{r}}^{\alpha_{r}}$ with $g_{i,\alpha_{1},...,\alpha_{r}} \in \mathbb{C}[x_{1},...,x_{n}]$. After expanding the expression $1+\sum_{i=1}^{n}(\partial_{i}-\sum_{k=1}^{r}\partial_{i}f_{k}\partial_{t_{k}})(\sum g_{i,\alpha_{1},...,\alpha_{r}}\partial_{t_{1}}^{\alpha_{1}}...\partial_{t_{r}}^{\alpha_{r}})$ and applying Lemma \ref{maximal_ideal}, we can get the following partial differential equations of $g_{i,\alpha_{1},...,\alpha_{r}}$.
	\begin{flalign*}
	\sum_{i=1}^{n}\partial_{i}g_{i,0,...,0}&=-1,   \\
	\sum_{i=1}^{n}\partial_{i}g_{i,\alpha_{1},...,\alpha_{r}}&=\sum_{i=1}^{n}\sum_{k=1}^{r}\partial_{i}f_{k} \cdot g_{i,\alpha_{1},...,\alpha_{k}-1,...,\alpha_{r}}(\alpha_{1}+...+\alpha_{r} \ge 1).
	\end{flalign*}
	In the notations above, $\alpha_{1},...,\alpha_{r} \in \mathbb{Z}_{\ge 0}$ and we set $g_{i,\alpha_{1},...,\alpha_{r}}=0$ if any $\alpha_{j}<0$. Thus we have the following proposition.
	
	\begin{proposition}\label{PDE}
	If the following partial differential equations
	\begin{flalign*}
		\sum_{i=1}^{n}\partial_{i}g_{i,0,...,0}&=-1,   \\
		\sum_{i=1}^{n}\partial_{i}g_{i,\alpha_{1},...,\alpha_{r}}&=\sum_{i=1}^{n}\sum_{k=1}^{r}\partial_{i}f_{k} \cdot g_{i,\alpha_{1},...,\alpha_{k}-1,...,\alpha_{r}}(\alpha_{1}+...\alpha_{r} \ge 1)
	\end{flalign*}
	do not have finite polynomial solutions, then $b_{I}(-\frac{n}{d})=0$. A finite polynomial solution means that all $g_{i,\alpha_{1},...,\alpha_{r}} \in \mathbb{C}[x_{1},...,x_{n}]$ and there exists an integer $N$ such that for all $\alpha_{1}+...\alpha_{r}>N$, $g_{i,\alpha_{1},...,\alpha_{r}}=0$.
	\end{proposition}

	One can use the Fischer pairing to change this set of equations to the existence of the other one. Firstly we introduce the Fischer pairing.
	
	\begin{definition}
     We define the Fischer pairing as follows.
	\begin{flalign*}
		\langle \cdot, \cdot \rangle: \mathbb{C}[x_{1},...,x_{n}] \times \mathbb{C}[x_{1},...,x_{n}] &\rightarrow \mathbb{C}  \\
		(f(x),g(x)) &\mapsto f(\partial) \cdot g(x)|_{x=0},
	\end{flalign*}
	where $f(\partial)$ denotes the differential operator obtained by replacing every $x_{i}$ with the corresponding $\partial_{i}$. 
	\end{definition}
	Since $\mathbb{C}[x_{1},...,x_{n}]$ is a $\mathbb{C}$-vector space generated by all monomials, it suffices to see how this pairing acts on monomials. By definition, we have
	\begin{flalign*}
		\langle x^{\alpha},x^{\beta} \rangle=\begin{cases}
			\alpha!   \ \mathrm{if} \ \alpha=\beta,   \\
			0   \  \mathrm{if} \ \alpha \neq \beta,
		\end{cases}
	\end{flalign*}
	where $\alpha=(\alpha_{1},...,\alpha_{n}),\beta=(\beta_{1},...,\beta_{n})$ are vectors, $x^{\alpha}=x_{1}^{\alpha_{1}}...x_{n}^{\alpha_{n}}$ and $\alpha !=\alpha_{1}!\alpha_{2}!...\alpha_{n}!$.
	
	We give some properties of the Fischer pairing within the following lemma. They can be proved easily using the characterization of the pairing above.
	
	\begin{lemma}\label{Fischer_pairing}
	The Fischer pairing has the following properties.
	
	$(1)$ This pairing is commutative, $\mathbb{C}$-bilinear and non-degenerate.
	
	$(2)$ The adjoint of some polynomial can be translated into the action of a differential operator, that is, for $f,g,h \in \mathbb{C}[x_{1},...,x_{n}]$, we have $\langle h(x)f(x),g(x) \rangle = \langle f(x), h(\partial) \cdot g(x)\rangle$.
	
	\end{lemma}

	Now we can prove the following proposition, which gives the dual of the original partial differential equation.
	\begin{proposition}\label{PDE_dual}
	The partial differential equations 
	\begin{flalign*}
		\sum_{i=1}^{n}\partial_{i}g_{i,0,...,0}&=-1,   \\
		\sum_{i=1}^{n}\partial_{i}g_{i,\alpha_{1},...,\alpha_{r}}&=\sum_{i=1}^{n}\sum_{k=1}^{r}\partial_{i}f_{k} \cdot g_{i,\alpha_{1},...,\alpha_{k}-1,...,\alpha_{r}}(\alpha_{1}+...+\alpha_{r} \ge 1)
	\end{flalign*}
	do not have finite polynomial solutions if and only if there exist polynomials $P_{\alpha_{1},...,\alpha_{r}}$ indexed by $\alpha_{1},...,\alpha_{r} \in \mathbb{Z}_{\ge 0}$ satisfying the following conditions.
	
	$(1)$ $P_{\alpha_{1},...,\alpha_{r}}$ is homogeneous of degree $d \cdot (\alpha_{1}+...+\alpha_{r})$,
	
	$(2)$ $P_{0,...,0} \neq 0$,
	
	$(3)$ $x_{i}P_{\alpha_{1},...,\alpha_{r}}(x)=\sum_{k=1}^{r}\partial_{i}f_{k}(\partial) \cdot P_{\alpha_{1},...,\alpha_{k}+1,...,\alpha_{r}}(x)$, where $\partial_{i}f_{k}(\partial)(1 \le k \le r)$ denotes the differential operator obtained by replacing every $x_{j}$ with the corresponding $\partial_{j}$. 
	\end{proposition}
	\begin{proof}
	We first prove the "if" part. Let $\Phi_{\alpha_{1},...,\alpha_{r}}=\sum_{i=1}^{n}\partial_{i}g_{i,\alpha_{1},...,\alpha_{r}}$. Then we have $\Phi_{0,..,0}=-1$ and $\Phi_{\alpha_{1},...,\alpha_{r}}=\sum_{i=1}^{n}\sum_{k=1}^{r}\partial_{i}f_{k} \cdot g_{i,\alpha_{1},...,\alpha_{k}-1,...,\alpha_{r}}$. We will prove $\sum_{\alpha_{1}+...+\alpha_{r}=M}\langle P_{\alpha_{1},...,\alpha_{r}}(x),\Phi_{\alpha_{1},...,\alpha_{r}}(x) \rangle$ is non-zero for every $M \ge 0$ by induction on $M$. When $M=0$, this forces $\alpha_{1}=...=\alpha_{r}=0$ and the conclusion follows from the definition. Assume the result holds for $M$.
	By Lemma \ref{Fischer_pairing}, we can calculate as follows.
	\begin{flalign*}
		\langle P_{\alpha_{1},...,\alpha_{r}}(x),\Phi_{\alpha_{1},...,\alpha_{r}}(x) \rangle=&\langle P_{\alpha_{1},...,\alpha_{r}}(x), \sum_{i=1}^{n}\partial_{i}g_{i,\alpha_{1},...,\alpha_{r}}\rangle   \\
		=&\sum_{i=1}^{n}\langle x_{i}P_{\alpha_{1},...,\alpha_{r}}(x), g_{i,\alpha_{1},...,\alpha_{r}} \rangle \\
		=&\sum_{i=1}^{n}\langle \sum_{k=1}^{r}\partial_{i}f_{k}(\partial) \cdot P_{\alpha_{1},...,\alpha_{k}+1,...,\alpha_{r}}(x), g_{i,\alpha_{1},...,\alpha_{r}} \rangle  \\
		=&\sum_{k=1}^{r}\langle P_{\alpha_{1},...,\alpha_{k}+1,...,\alpha_{r}}(x),\sum_{i=1}^{n}\partial_{i}f_{k}(x)g_{i,\alpha_{1},...,\alpha_{r}} \rangle.
	\end{flalign*}
	We have
	\begin{flalign*}
		&\sum_{\alpha_{1}+...+\alpha_{r}=M}\langle P_{\alpha_{1},...,\alpha_{r}}(x),\Phi_{\alpha_{1},...,\alpha_{r}}(x) \rangle   \\
		=&\sum_{\alpha_{1}+...+\alpha_{r}=M}\sum_{k=1}^{r}\langle P_{\alpha_{1},...,\alpha_{k}+1,...,\alpha_{r}}(x),\sum_{i=1}^{n}\partial_{i}f_{k}(x)g_{i,\alpha_{1},...,\alpha_{r}} \rangle  \\
		=&\sum_{\alpha_{1}+...+\alpha_{r}=M+1}\langle P_{\alpha_{1},...,\alpha_{r}}(x), \sum_{k=1}^{r}\sum_{i=1}^{n}\partial_{i}f_{k}(x)g_{i,\alpha_{1},...,\alpha_{k}-1,...,\alpha_{r}} \rangle   \\
		=&\sum_{\alpha_{1}+...+\alpha_{r}=M+1}\langle P_{\alpha_{1},...,\alpha_{r}}(x),\Phi_{\alpha_{1},...,\alpha_{r}} \rangle.
	\end{flalign*}
	Thus $\sum_{\alpha_{1}+...+\alpha_{r}=M}\langle P_{\alpha_{1},...,\alpha_{r}}(x),\Phi_{\alpha_{1},...,\alpha_{r}}(x) \rangle$ is non-zero for every $M$. If the original partial differential equations regarding $g_{i,\alpha_{1},...,\alpha_{r}}$ have finite polynomial solutions, then there exists $N$ such that $g_{i,\alpha_{1},...,\alpha_{r}}=0$ whenever $\alpha_{1}+...+\alpha_{r} \ge N$. Now we take $M>N$ and $\sum_{\alpha_{1}+...+\alpha_{r}=M}\langle P_{\alpha_{1},...,\alpha_{r}}(x),\Phi_{\alpha_{1},...,\alpha_{r}}(x) \rangle$ must be zero, a contradiction, so the "if" part follows.

	Next we prove the "only if" part. To cut out polynomials in different degrees, we need to define an auxiliary function. Let $G_{i}=\sum_{\alpha_{1},...,\alpha_{r}}g_{i,\alpha_{1},...,\alpha_{r}}(x)t_{1}^{\alpha_{1}}...t_{r}^{\alpha_{r}}$ and $D_{i}=\partial_{i}-\sum_{k=1}^{r}t_{k}\partial_{i}f_{k}(x)$, then the original system of equations is equivalent to $\sum_{i=1}^{n}D_{i} \cdot G_{i}=-1$. It has finite polynomial solutions if and only if $G_{i}$ are polynomials. Let $\mathbb{C}[x_{1},...,x_{n}]_{l}$ denote the homogeneous polynomials of degree $l$. We consider the following map.
	\begin{flalign*}
	T:E:=(\bigoplus_{\alpha_{1},...,\alpha_{r}}\mathbb{C}[x_{1},...,x_{n}]_{d|\alpha|+1}t_{1}^{\alpha_{1}}...t_{r}^{\alpha_{r}})^{\oplus n} &\rightarrow F:=\bigoplus_{\alpha_{1},...,\alpha_{r}}\mathbb{C}[x_{1},...,x_{n}]_{d|\alpha|}t_{1}^{\alpha_{1}}...t_{r}^{\alpha_{r}}  \\
	(G_{1},...,G_{n}) &\mapsto \sum_{i=1}^{n}D_{i} \cdot G_{i}.	
	\end{flalign*}
	
	In $E,F$ we require that the elements have finite support in the big direct sum, that is, their elements are polynomials in $\mathbb{C}[x_{1},...,x_{n},t_{1},...,t_{r}]$.  Let Im($T$) be the image of the map $T$ and it is a linear subspace of $F$. Since the original equations do not have finite polynomial solutions, we have $-1 \notin \mathrm{Im}(T)$. We consider the Fischer pairing in $F$ with respect to the variables $x_{1},...,x_{n}$. Because the Fischer pairing is non-trivial only in the same degree, we can define $F^{*}:=\prod_{\alpha_{1},...,\alpha_{r}}\mathbb{C}[x_{1},...,x_{n}]_{d|\alpha|}t_{1}^{\alpha_{1}}...t_{r}^{\alpha_{r}}$. In $F^{*}$ we do not require that the elements have finite support. We consider the Fischer pairing between $F$ and $F^{*}$ as follows.
	\begin{flalign*}
		F^{*} \times F \rightarrow& \mathbb{C}  \\
		(\sum_{\alpha_{1},...,\alpha_{r}}f_{\alpha_{1},...,\alpha_{r}}t_{1}^{\alpha_{1}}...t_{r}^{\alpha_{r}},\sum_{\alpha_{1},...,\alpha_{r}}g_{\alpha_{1},...,\alpha_{r}}t_{1}^{\alpha_{1}}...t_{r}^{\alpha_{r}})\mapsto& \sum_{\alpha_{1},...,\alpha_{r}}\langle f_{\alpha_{1},...,\alpha_{r}}, g_{\alpha_{1},...,\alpha_{r}} \rangle.
	\end{flalign*}

	 Since the Fischer pairing is also non-degenerate at every homogeneous level, for any linear map $\phi \in \mathrm{Hom}_{\mathbb{C}}(F,\mathbb{C})$, there exists $P \in F^{*}$ such that $\phi=\langle P,\bullet \rangle$. Because $-1 \notin \mathrm{Im}(T)$, there exists $P \in F^{*}$ such that $\langle P,-1 \rangle \neq 0$ and $\langle P,\mathrm{Im}(T) \rangle=0$. Suppose $P=\sum_{\alpha_{1},...,\alpha_{r}}P_{\alpha_{1},...,\alpha_{r}}t_{1}^{\alpha_{1}}...t_{r}^{\alpha_{r}}$, where $P_{\alpha_{1},...,\alpha_{r}}$ are polynomials of degree $d|\alpha|$, then we have
	\begin{flalign*}
		0 \neq& \langle P,-1 \rangle  \\
		=&-P_{0,...,0}
	\end{flalign*}
	and
	\begin{flalign*}
		0=& \langle P,\sum_{i=1}^{n}D_{i} \cdot G_{i} \rangle \\
		=&\sum_{i=1}^{n} \langle \sum_{\alpha_{1},...,\alpha_{r}}P_{\alpha_{1},...,\alpha_{r}}t_{1}^{\alpha_{1}}...t_{r}^{\alpha_{r}},\partial_{i}g_{i,\alpha_{1},...,\alpha_{r}}t_{1}^{\alpha_{1}}...t_{r}^{\alpha_{r}}-\sum_{k=1}^{r}\partial_{i}f_{k}(x)g_{i,\alpha_{1},...,\alpha_{r}}t_{1}^{\alpha_{1}}...t_{k}^{\alpha_{k}+1}...t_{r}^{\alpha_{r}} \rangle  \\
		=&\sum_{i=1}^{n}\sum_{\alpha_{1},...,\alpha_{r}}\langle P_{\alpha_{1},...,\alpha_{r}},\partial_{i}g_{i,\alpha_{1},...,\alpha_{r}} \rangle-\langle P_{\alpha_{1},...,\alpha_{r}},\sum_{k=1}^{r}\partial_{i}f_{k}(x)g_{i,\alpha_{1},...,\alpha_{k}-1,...,\alpha_{r}} \rangle \\
		=&\sum_{i=1}^{n}\sum_{\alpha_{1},...,\alpha_{r}} \langle x_{i}P_{\alpha_{1},...,\alpha_{r}}-\sum_{k=1}^{r}\partial_{i}f_{k}(\partial) \cdot P_{\alpha_{1},...,\alpha_{k}+1,...,\alpha_{r}}, g_{i,\alpha_{1},...,\alpha_{r}} \rangle.
	\end{flalign*}
	for all possible $g_{i,\alpha_{1},...,\alpha_{r}}$. This tells us $P_{0,...,0} \neq 0$ and $x_{i}P_{\alpha_{1},...,\alpha_{r}}-\sum_{k=1}^{r}\partial_{i}f_{k}(\partial) \cdot P_{\alpha_{1},...,\alpha_{k}+1,...,\alpha_{r}}=0$, so the result follows.
	\end{proof}
	
	Combining Proposition \ref{PDE} and Proposition \ref{PDE_dual}, we have the following theorem.
	\begin{theorem}\label{PDE_theorem}
	For ideal $I=(f_{1},...,f_{r}) \subset \mathbb{C}[x_{1},...,x_{n}]$ with homogeneous generators of degree $d$, suppose that there exist polynomials $P_{\alpha_{1},...,\alpha_{r}}$ indexed by $\alpha_{1},...,\alpha_{r} \in \mathbb{Z}_{\ge 0}$ satisfying the following conditions.
	
	$(1)$ $P_{\alpha_{1},...,\alpha_{r}}$ is homogeneous of degree $d \cdot (\alpha_{1}+...+\alpha_{r})$,
	
	$(2)$ $P_{0,...,0} \neq 0$,
	
	$(3)$ $x_{i}P_{\alpha_{1},...,\alpha_{r}}(x)=\sum_{k=1}^{r}\partial_{i}f_{k}(\partial) \cdot P_{\alpha_{1},...,\alpha_{k}+1,...,\alpha_{r}}(x)$.
	
	Then $b_{I}(-\frac{n}{d})=0$.
	\end{theorem}
	
	\begin{remark}
		The set of pde above can be viewed as a generalized Airy system.  
		One can solve this system, in the usual analytic or distributional sense, by a Fourier--Laplace transform. But the situation changes substantially when one asks for polynomial solutions.
	\end{remark}
	
	\begin{remark}\label{Auxiliary_function}
	One can use auxiliary functions to write the condition of Theorem \ref{PDE_theorem} in a simpler form. Concretely, let $\mathcal{P}=\sum_{\alpha_{1},...,\alpha_{r}}P_{\alpha_{1},...,\alpha_{r}}\frac{t_{1}^{\alpha_{1}}...t_{r}^{\alpha_{r}}}{\alpha_{1}!...\alpha_{r}!}$ be a power series in $t_{1},...,t_{r}$ with coefficient in $\mathbb{C}[x_{1},...,x_{n}]$, then the condition $(3)$ in \ref{PDE_theorem} is equivalent to
	\begin{align*}
		x_{i}\mathcal{P}=\sum_{k=1}^{r}\partial_{i}f_{k}(\partial) \cdot \partial_{t_{k}}\mathcal{P}
	\end{align*}
	for any $1 \le i \le n$. 
	\end{remark}
	
	\section{Applications and Stability Properties}\label{sec4}
	In this section, we give some applications of Theorem \ref{PDE_theorem}. Before that, we give a definition characterizing the condition in Theorem \ref{PDE_theorem}.
	
	\begin{definition}
	Suppose $f_{1},...,f_{r} \in \mathbb{C}[x_{1},...,x_{n}]$ are polynomials of degree $d$. We say these polynomials have the polynomial PDE property if they satisfy the condition in Theorem \ref{PDE_theorem}, that is, there exist polynomials $P_{\alpha_{1},...,\alpha_{r}}$ indexed by $\alpha_{1},...,\alpha_{r} \in \mathbb{Z}_{\ge 0}$ satisfying the following conditions.
	
	$(1)$ $P_{\alpha_{1},...,\alpha_{r}}$ is homogeneous of degree $d \cdot (\alpha_{1}+...+\alpha_{r})$,
	
	$(2)$ $P_{0,...,0} \neq 0$,
	
	$(3)$ $x_{i}P_{\alpha_{1},...,\alpha_{r}}(x)=\sum_{k=1}^{r}\partial_{i}f_{k}(\partial) \cdot P_{\alpha_{1},...,\alpha_{k}+1,...,\alpha_{r}}(x)$.
	\end{definition}
	
	In the following subsections, we will give examples which satisfy the polynomial PDE property.

	\subsection{Isolated singularity case}
	A simple case is $r=1$ and $f_{1}$ has an isolated singularity at $0$. We have the following proposition.
	\begin{proposition}
		Suppose $r=1$, $f_{1}=f$ has an isolated singularity at $0$. Then $f$ has the polynomial PDE property.
	\end{proposition}
	\begin{proof}
		We will denote the required polynomials by $Q_{N}(N \in \mathbb{Z}_{\ge 0})$. We construct these $Q_{N}$ by induction. For $N=0$, we set $Q_{0}=1$. Set $Q_{-1}=0$. Since $d\ge2$ and $Q_0=1$, we have $(\partial_jf)(\partial)Q_0=0=x_jQ_{-1}$ for every $j$. Hence the identity $(\partial_jf)(\partial)Q_K=x_jQ_{K-1}$ also holds for $K=0$. Now suppose we have $Q_{0},...,Q_{K}$ satisfying the conditions in Theorem \ref{PDE_theorem}. We construct $Q_{K+1}$ such that $Q_{K+1}$ is homogeneous of degree $d(K+1)$ and $x_{i}Q_{K}=\partial_{i}f(\partial) \cdot Q_{K+1}(x)$. We consider the following linear map.
		\begin{flalign*}
			T:	\mathbb{C}[x_{1},...,x_{n}]_{d(K+1)} &\rightarrow (\mathbb{C}[x_{1},...,x_{n}]_{dK+1})^{\oplus n}  \\
			g  &\mapsto  (\partial_{1}f(\partial) \cdot g,...,\partial_{n}f(\partial) \cdot g).
		\end{flalign*}		
		It suffices to show that $(x_{1}Q_{K},...,x_{n}Q_{K}) \in \mathrm{Im}(T)$. Since the Fischer pairing is non-degenerate, we only need to prove that for any $(h_{1},...,h_{n}) \in (\mathbb{C}[x_{1},...,x_{n}]_{dK+1})^{\oplus n}$ satisfying $\sum_{i=1}^{n}\langle h_{i},\partial_{i}f(\partial) \cdot g \rangle=0$ for all $g$, we have $\sum_{i=1}^{n} \langle h_{i},x_{i}Q_{K} \rangle=0$. Note that $\sum_{i=1}^{n}\langle h_{i},\partial_{i}f(\partial) \cdot g \rangle=\sum_{i=1}^{n}\langle \partial_{i}f(x)h_{i}, g \rangle$, so $\sum_{i=1}^{n}\partial_{i}f(x)h_{i}=0$. Because the polynomials $\partial_{1}f(x),...,\partial_{n}f(x)$ form a regular sequence, their Koszul complex is exact, so there exist polynomials $b_{ij}$ such that $b_{ij}=-b_{ji}$ and $h_{i}=\sum_{j=1}^{n}b_{ij}\partial_{j}f$. Now we can calculate as follows.
		\begin{flalign*}
			&\sum_{i=1}^{n} \langle h_{i},x_{i}Q_{K} \rangle  \\
			=&\sum_{i=1}^{n} \langle \sum_{j=1}^{n}b_{ij}\partial_{j}f,x_{i}Q_{K} \rangle  \\
			=&\sum_{i=1}^{n}\sum_{j=1}^{n} \langle b_{ij},\partial_{j}f(\partial) \cdot (x_{i}Q_{K}) \rangle.
		\end{flalign*}
		This tells us that we only need to prove $\partial_{j}f(\partial) \cdot (x_{i}Q_{K})=\partial_{i}f(\partial) \cdot (x_{j}Q_{K})$. In fact, we have
		\begin{flalign*}
			&\partial_{j}f(\partial) \cdot (x_{i}Q_{K}) \\
			=&x_{i}\partial_{j}f(\partial) \cdot Q_{K}+\partial_{i}\partial_{j}f(\partial)Q_{K} \\
			=&x_{i}x_{j}Q_{K-1}+\partial_{i}\partial_{j}f(\partial)Q_{K}.
		\end{flalign*}
		Similarly $\partial_{i}f(\partial) \cdot (x_{j}Q_{K})$ also equals $x_{i}x_{j}Q_{K-1}+\partial_{i}\partial_{j}f(\partial)Q_{K}$. Thus we have $\partial_{j}f(\partial) \cdot (x_{i}Q_{K})=\partial_{i}f(\partial) \cdot (x_{j}Q_{K})$, and the proposition follows.
	\end{proof}

	\subsection{Maximal-Minor Determinantal Ideals}
	Let $M$ be an $m \times n(m \le n)$ matrix with variables $(x_{ij})_{m \times n}$ and $S=\mathbb{C}[x_{ij}]_{m \times n}$ be its coordinate ring. Let $I_{r}$ be the ideal generated by $r \times r$ minors of the matrix and $Z_{r} \subset \mathbb{C}^{mn}$ be the sub-scheme defined by $I_{r}$.

	\begin{remark}\label{SMC_for_det}
	If one can prove the polynomial PDE property for general $I_{r}$, then the strong monodromy conjecture for $I_{r}$ holds. The reason is the following. 	To distinguish different embedded spaces, we denote the sub-variety defined by $u \times u$ minors in the space of $v \times w$ and the corresponding ideal by $Z_{u}^{v,w}$ and $I_{u}^{v,w}$ for positive integers $u \le v \le w$. Let $U_{ij}=\{x_{ij} \neq 0\} \subset \mathbb{C}^{mn}$ for every $i,j$. Using the same argument as in Lemma $2.3$ in \cite{Bernstein_for_maximal_minors}, one can prove the following isomorphism.
	\begin{flalign*}
		U_{ij} \cap Z_{r}^{m,n} \cong \mathbb{C}^{*} \times \mathbb{C}^{m-1} \times \mathbb{C}^{n-1} \times Z_{r-1}^{m-1,n-1}(r \ge 2).
	\end{flalign*}
	Since $\mathbb{C}^{*} \times \mathbb{C}^{m-1} \times \mathbb{C}^{n-1}$ is a smooth affine variety, by Proposition \ref{multiplying_smooth_variety} we have $b_{Z_{r-1}^{m-1,n-1}}(s)=b_{U_{ij}\cap Z_r^{m,n}}(s)$. By Proposition \ref{open_subset_b_function}, $b_{Z_{r-1}^{m-1,n-1}}|b_{Z_{r}^{m,n}}$. Note that using the isomorphism above we can get $\mathrm{codim}_{\mathbb{C}^{mn}}(Z_{r}^{m,n})=\mathrm{codim}_{\mathbb{C}^{(m-1)(n-1)}}(Z_{r-1}^{m-1,n-1})$, so $b_{I_{r-1}^{m-1,n-1}}|b_{I_{r}^{m,n}}$. Similarly, we have $b_{I_{r-k}^{m-k,n-k}}|b_{I_{r-k+1}^{m-k+1,n-k+1}}$ for every integer $1 \le k \le r-1$.
	
	If we can prove $b_{I_{r}^{m,n}}(-\frac{mn}{r})=0$, then by the same method we have $b_{I_{r-k}^{m-k,n-k}}(-\frac{(m-k)(n-k)}{r-k})=0$ for $0 \le k \le r-1$. By the division relations of these Bernstein-Sato polynomials above we know $-\frac{(m-k)(n-k)}{r-k}$ is a root for $0 \le k \le r-1$. 
	
	On the other hand, by \cite{CZ25}, the poles of the motivic zeta functions are exactly $-\frac{(m-k)(n-k)}{r-k}$ for $0 \le k \le r-1$, so the strong monodromy conjecture for $I_{r}$ holds.
	\end{remark}

	In this subsection, we verify the polynomial PDE property for the case when $r=m$, that is, the case of maximal minors. We first fix our notations. For any set $A \subset \{1,...,m\}$ and any set $B \subset \{1,...,n\}$ satisfying $|A|=|B|=r$, we denote the minor generated by the rows and columns indexed by $A$ and $B$ respectively by $f_{A,B}$. In the case of maximal minors, $A$ has to be $\{1,...,m\}$, so we will simply denote $f_{A,B}$ by $f_{B}$. Similarly, for the indexes $\alpha$, we will denote the index corresponding to $f_{B}$ by $\alpha_{B}$ and let $\vec{\alpha}=(\alpha_{B})_{B \subset \{1,...,n\}}$. Now we take the required polynomials in Theorem \ref{PDE_theorem} as $P_{\vec{\alpha}}=c_{\alpha}\prod_{B \subset \{1,...,n\}}f_{B}^{\alpha_{B}}$, where $c_{\alpha} \in \mathbb{C}$ is non-zero. It suffices to prove the following proposition to use our theorem.
	
	\begin{proposition}\label{identity_for_minors}
	In the notation above, we have $$\prod_{k=1}^{m-1}(|\vec{\alpha}|+n-k)x_{ij}\prod_{B \subset \{1,...,n\}}f_{B}^{\alpha_{B}}=\sum_{A \subset \{1,...,n\}}\partial_{ij}f_{A}(\partial) \cdot f_{A}\prod_{B \subset \{1,...,n\}}f_{B}^{\alpha_{B}}.$$ 
	\end{proposition}
	\begin{proof}
	Let $F_{\vec{\alpha}}=\prod_{B \subset \{1,...,n\}}f_{B}^{\alpha_{B}}$, and let $C=\sum_{A \subset \{1,...,n\}}f_{A}(\partial)f_{A}(x)$ be a differential operator. For any $x_{ij}$, we have
	\begin{flalign*}
		[C,x_{ij}]&=Cx_{ij}-x_{ij}C \\
		&=\sum_{A \subset \{1,...,n\}}f_{A}(\partial)f_{A}(x)x_{ij}-x_{ij}f_{A}(\partial)f_{A}(x)  \\
		&=\sum_{A \subset \{1,...,n\}}(f_{A}(\partial)x_{ij}-x_{ij}f_{A}(\partial))f_{A}(x) \\
		&=\sum_{A \subset \{1,...,n\}}\partial_{ij}f_{A}(\partial)f_{A}(x). 
	\end{flalign*}
	This implies that the right hand side can be written as
	\begin{align}\label{equation1}
		&\sum_{A \subset \{1,...,n\}}\partial_{ij}f_{A}(\partial) f_{A}\prod_{B \subset \{1,...,n\}}f_{B}^{\alpha_{B}}   \\
		=&Cx_{ij}F_{\vec{\alpha}}-x_{ij}CF_{\vec{\alpha}}.
	\end{align}
	Next we calculate $CF_{\vec{\alpha}}$ and $Cx_{ij}F_{\vec{\alpha}}$ respectively. To do this, we shall use the following formula by Corollary $1.3$ in \cite{Capelli_formula}.
	\begin{align}\label{Cauchy_Binet}
		C=\mathrm{cdet}(E+\mathrm{diag}(n,n-1,...,n-m+1))_{m \times m},
	\end{align}
	where $E$ is a $m \times m$ matrix with entries $E_{ab}=\sum_{q=1}^{n}x_{aq}\partial_{bq}$ and cdet is the column determinant, that is, for a $m \times m$ matrix $D$ with entries $D_{ij}$, $$\mathrm{cdet}(D)=\sum_{\sigma \in S_{m}}\mathrm{sgn}(\sigma)D_{\sigma(1)1}...D_{\sigma(m)m}.$$
	
	Now we will use \ref{Cauchy_Binet} to calculate $CF_{\vec{\alpha}}$ and $Cx_{ij}F_{\vec{\alpha}}$. According to the definition of column determinant, we have
	\begin{flalign*}
		C=\sum_{\sigma \in S_{m}}\mathrm{sgn}(\sigma)\prod_{k=1}^{m}(E_{\sigma(k)k}+\delta_{\sigma(k)k}(n-k+1)).
	\end{flalign*}
	Note that when $E_{ab}$ acts on $f_{A}$, it replaces its $b$-th row by its $a$-th row. This implies that for any $A \subset \{1,...,n\}$, $E_{ab} \cdot f_{A}=\delta_{ab}f_{A}$, so $E_{ab} \cdot F_{\vec{\alpha}}=\delta_{ab}|\vec{\alpha}|F_{\vec{\alpha}}$. For any permutation $\sigma$, $\prod_{k=1}^{m}(E_{\sigma(k)k}+\delta_{\sigma(k)k}(n-k+1)) \cdot F_{\vec{\alpha}}$ is non-zero if and only if $\sigma=\mathrm{id}$, so
	\begin{align}\label{equation2}
	CF_{\vec{\alpha}}=\prod_{k=1}^{m}(|\vec{\alpha}|+n-k+1)F_{\vec{\alpha}}.
	\end{align} 
	
	Next we calculate $Cx_{ij}F_{\vec{\alpha}}$. As in the case above, we have
	\begin{flalign*}
		&E_{ab}x_{ij}F_{\vec{\alpha}}  \\
		=&\sum_{q=1}^{n}x_{aq}\partial_{bq}x_{ij}F_{\vec{\alpha}}  \\
		=&\sum_{q=1}^{n}x_{aq}(x_{ij}\partial_{bq}+\delta_{bi}\delta_{qj})F_{\vec{\alpha}}  \\
		=&x_{ij}E_{ab}F_{\vec{\alpha}}+\delta_{bi}x_{aj}F_{\vec{\alpha}} \\
		=&|\vec{\alpha}|\delta_{ab}x_{ij}F_{\vec{\alpha}}+\delta_{bi}x_{aj}F_{\vec{\alpha}}.
	\end{flalign*}
	We firstly consider the case when $i=1$, that is, $Cx_{1j}F_{\vec{\alpha}}$. By the computation above, if $a \neq b$ and $b \neq 1$, then $E_{ab}x_{1j}F_{\vec{\alpha}}=0$. Moreover, if $b \neq 1$, $E_{ab}x_{1j}F_{\vec{\alpha}}$ is still a multiple of $x_{1j}F_{\vec{\alpha}}$. This implies that if $\prod_{k=1}^{m}(E_{\sigma(k)k}+\delta_{\sigma(k)k}(n-k+1)) \cdot x_{1j}F_{\vec{\alpha}}$ is non-zero, then $\sigma=\mathrm{id}$. Thus we have
	\begin{flalign*}
		Cx_{1j}F_{\vec{\alpha}}=(|\vec{\alpha}|+n+1)\prod_{k=2}^{m}(|\vec{\alpha}|+n-k+1)x_{1j}F_{\vec{\alpha}}.
	\end{flalign*}
	Since these variables are symmetric, one can prove other cases after a linear coordinate change. Thus we have
	\begin{align}\label{equation3}
		Cx_{ij}F_{\vec{\alpha}}=(|\vec{\alpha}|+n+1)\prod_{k=2}^{m}(|\vec{\alpha}|+n-k+1)x_{ij}F_{\vec{\alpha}}.
	\end{align}
	Combining equations \ref{equation1}, \ref{equation2} and \ref{equation3}, we have
	\begin{flalign*}
		&\sum_{A \subset \{1,...,n\}}\partial_{ij}f_{A}(\partial) f_{A}\prod_{B \subset \{1,...,n\}}f_{B}^{\alpha_{B}}   \\
		=&Cx_{ij}F_{\vec{\alpha}}-x_{ij}CF_{\vec{\alpha}}  \\
		=&(|\vec{\alpha}|+n+1)\prod_{k=2}^{m}(|\vec{\alpha}|+n-k+1)x_{ij}F_{\vec{\alpha}}-x_{ij}\prod_{k=1}^{m}(|\vec{\alpha}|+n-k+1)F_{\vec{\alpha}} \\
		=&\prod_{k=1}^{m-1}(|\vec{\alpha}|+n-k)x_{ij}F_{\vec{\alpha}}.
	\end{flalign*}
	\end{proof}
	
	\begin{remark}
	When $m=n$, we have $r=1$, and the required equation in the proposition above is
	\begin{flalign*}
		\prod_{k=1}^{m-1}(\alpha+m-k)x_{ij}f^{\alpha}=\partial_{ij}f(\partial) \cdot f^{\alpha+1},
	\end{flalign*}
	where $f=f_{1}$ is the determinant of the matrix $(x_{ij})_{m \times m}$. This is actually a variant of the famous Cayley identity.
	\end{remark}

	Now we can prove the $\frac{n}{d}$ property of the ideal of maximal minors.
	\begin{theorem}\label{n/d_of_maximal_minors}
	Suppose $I_{m}$ is the ideal generated by maximal minors. Then $b_{I_{m}}(-\frac{mn}{m})=b_{I_{m}}(-n)=0$.
	\end{theorem}
	\begin{proof}
	We take $P_{\vec{\alpha}}=c_{\vec{\alpha}}\prod_{B \subset \{1,...,n\}}f_{B}^{\alpha_{B}}$, where $c_{\vec{\alpha}}=\prod_{k=1}^{m-1}\frac{(n-k-1)!}{(n-k+|\vec{\alpha}|-1)!}$. Let us verify that this choice satisfies the conditions of Theorem \ref{PDE_theorem}. The conditions $(1)$ and $(2)$ are straightforward, and the condition $(3)$ follows from Proposition \ref{identity_for_minors}, so the result holds.
	\end{proof}

	\subsection{Stability under Algebraic Operations}
	In this subsection, we will prove the polynomial PDE property for some $g_{1},...,g_{l}$ when we know the polynomial PDE property for the other set of polynomials $f_{1},...,f_{r}$. The theorem is as follows.
	
	\begin{theorem}\label{linear_span}
	Suppose $g_{1},...,g_{l},f_{1},...,f_{r} \in \mathbb{C}[x_{1},...,x_{n}]$ are homogeneous polynomials of degree $d$ satisfying $\mathrm{Span}_{\mathbb{C}}\{g_{1},...,g_{l}\} \subset \mathrm{Span}_{\mathbb{C}}\{f_{1},...,f_{r}\}$, that is, $g_{1},...,g_{l}$ are linear combinations of $f_{1},...,f_{r}$. If $g_{1},...,g_{l}$ have the polynomial PDE property, then $f_{1},...,f_{r}$ also have the polynomial PDE property.
	\end{theorem}
	\begin{proof}
	We write $g_{a}=\sum_{k=1}^{r}c_{ak}f_{k}$, where $c_{ak} \in \mathbb{C}$. Since $g_{1},...,g_{l}$ have the polynomial PDE property, by Remark \ref{Auxiliary_function}, there exists $\mathcal{Q}(s_{1},...,s_{l})=\sum_{\beta_{1},...,\beta_{l}}Q_{\beta_{1},...,\beta_{l}}\frac{s_{1}^{\beta_{1}}...s_{l}^{\beta_{l}}}{\beta_{1}!...\beta_{l}!}$ satisfying the condition in Theorem \ref{PDE_theorem}, where $Q_{\beta_{1},...,\beta_{l}} \in \mathbb{C}[x_{1},...,x_{n}]$. Since $g_{a}=\sum_{k=1}^{r}c_{ak}f_{k}$, we have $\partial_{i}g_{a}(\partial)=\sum_{k=1}^{r}c_{ak}\partial_{i}f_{k}(\partial)$. We take $s_{a}=\sum_{k=1}^{r}c_{ak}t_{k}$, and we define $\mathcal{P}(t_{1},...,t_{r})=\mathcal{Q}(s_{1},...,s_{l})$. By the chain rule we have $\partial_{t_{k}}=\sum_{a=1}^{l}\partial_{t_{k}}s_{a} \partial_{s_{a}}=\sum_{a=1}^{l}c_{ak}\partial_{s_{a}}$.
	This implies that
	\begin{flalign*}
		x_{i}\mathcal{P}=&x_{i}\mathcal{Q}  \\
		                =&\sum_{a=1}^{l}\partial_{i}g_{a}(\partial) \cdot \partial_{s_{a}}\mathcal{Q}  \\
		                =&\sum_{a=1}^{l}\sum_{k=1}^{r}c_{ak}\partial_{i}f_{k}(\partial)\cdot \partial_{s_{a}}\mathcal{Q} \\
		                =&\sum_{k=1}^{r}\partial_{i}f_{k}(\partial)\sum_{a=1}^{l}c_{ak}\partial_{s_{a}}\mathcal{Q}  \\
		                =&\sum_{k=1}^{r}\partial_{i}f_{k}(\partial)\partial_{t_{k}}\mathcal{P}.
	\end{flalign*}
	The conditions $(1)$ and $(2)$ in Theorem \ref{PDE_theorem} can be easily verified from the definition of $\mathcal{P}$, so the result follows.
	\end{proof}	
	
	\begin{corollary}
	Suppose $f_{1},...,f_{r},f_{r+1},...,f_{r+l}$ are homogeneous polynomials of degree $d$. If $f_{1},...,f_{r}$ have the polynomial PDE property, then $f_{1},...,f_{r+l}$ also have the polynomial PDE property.
	\end{corollary}
	
	\begin{corollary}
	Suppose $f_{1},...,f_{r}$ are homogeneous polynomials of degree $d$. If any of the $f_{i}$ has the polynomial PDE property as a single polynomial, for example, $f_{i}$ has an isolated singularity at $0$, then $f_{1},...,f_{r}$ have the polynomial PDE property.
	\end{corollary}
	
	\begin{corollary}
	Suppose $f_{1},...,f_{r}$ are homogeneous polynomials of degree $d$ and there exist $\lambda_{1},...,\lambda_{r} \in \mathbb{C}$ such that $\lambda_{1}f_{1}+...+\lambda_{r}f_{r}$ has isolated singularity at $0$. Then $f_{1},...,f_{r}$ have the polynomial PDE property.
	\end{corollary}
		
	\begin{theorem}\label{ideal_sum}
	Assume $f_{1},....,f_{r} \in \mathbb{C}[x_{1},...,x_{n}]$ and $g_{1},...,g_{l} \in \mathbb{C}[y_{1},...,y_{m}]$ are homogeneous polynomials of degree $d$. If $f_{1},...,f_{r}$ and $g_{1},...,g_{l}$ both have the polynomial PDE property, then regarded as polynomials in $\mathbb{C}[x_{1},...,x_{n},y_{1},...,y_{m}]$, the set of polynomials $f_{1},...,f_{r},g_{1},...,g_{l}$ also has the polynomial PDE property.
	\end{theorem}
	\begin{proof}
	By Remark \ref{Auxiliary_function}, there exist $\mathcal{P}(t_{1},...,t_{r})=\sum_{\alpha_{1},...,\alpha_{r}}\mathcal{P}_{\alpha_{1},...,\alpha_{r}}\frac{t_{1}^{\alpha_{1}}...t_{r}^{\alpha_{r}}}{\alpha_{1}!...\alpha_{r}!}$ and $\mathcal{Q}(s_{1},...,s_{l})=\sum_{\beta_{1},...,\beta_{l}}\mathcal{Q}_{\beta_{1},...,\beta_{l}}\frac{s_{1}^{\beta_{1}}...s_{l}^{\beta_{l}}}{\beta_{1}!...\beta_{l}!}$ satisfying the conditions in Theorem \ref{PDE_theorem}. We take the auxiliary function for $f_{1},...,f_{r},g_{1},...,g_{l}$ as $\mathcal{R}(t_{1},...,t_{r},s_{1},...,s_{l})=\mathcal{P} \cdot \mathcal{Q}$. For any $x_{i}$, we have
	\begin{flalign*}
		x_{i}\mathcal{R}=&x_{i}\mathcal{P}\mathcal{Q} \\
		                =&\sum_{k=1}^{r}\partial_{i}f_{k}(\partial)\cdot \partial_{t_{k}}\mathcal{P} \mathcal{Q}  \\
		                =&\sum_{k=1}^{r}\partial_{i}f_{k}(\partial)\cdot \partial_{t_{k}}\mathcal{R}.
	\end{flalign*}
	Similarly, for any $y_{j}$, we have
	\begin{flalign*}
		y_{j}\mathcal{R}=&y_{j}\mathcal{Q}\mathcal{P} \\
		=&\sum_{a=1}^{l}\partial_{j}g_{a}(\partial)\cdot \partial_{s_{a}}\mathcal{Q} \mathcal{P}  \\
		=&\sum_{a=1}^{l}\partial_{j}g_{a}(\partial)\cdot \partial_{s_{a}}\mathcal{R}.
	\end{flalign*}
	Using these two equations above, one can verify that $\mathcal{R}$ has the required properties, so the conditions in Theorem \ref{PDE_theorem} are satisfied. 
	\end{proof}	
	
	The following several theorems will focus on the case of a single homogeneous polynomial $f$ of degree $d$. In this case, there exist polynomials $P_{\alpha}$ indexed by $\alpha \in \mathbb{Z}_{\ge 0}$ satisfying the following conditions.
	
	$(1)$ $P_{\alpha}$ is homogeneous of degree $d \cdot \alpha$,
	
	$(2)$ $P_{0} \neq 0$,
	
	$(3)$ $x_{i}P_{\alpha}(x)=\partial_{i}f(\partial) \cdot P_{\alpha+1}(x)$.
	
	From $(3)$ we have the following lemma.
	\begin{lemma}\label{f_partial}
	Under the assumptions above, we have $f(\partial) \cdot P_{\alpha+1}(x)=\frac{n+d\alpha}{d}P_{\alpha}(x)$.
	\end{lemma}
	\begin{proof}
	Since $f$ is homogeneous of degree $d$, we have the Euler identity
	\begin{flalign*}
		\sum_{i=1}^{n}x_{i}\partial_{i}f=d \cdot f.
	\end{flalign*}
	Replacing each $x_{i}$ by $\partial_{i}$, we have
	\begin{flalign*}
		\sum_{i=1}^{n}\partial_{i}(\partial_{i}f)(\partial)=d \cdot f(\partial).
	\end{flalign*}
	This implies that
	\begin{flalign*}
		&d  f(\partial) \cdot P_{\alpha+1}(x) \\
		=&\sum_{i=1}^{n}\partial_{i}(\partial_{i}f)(\partial) \cdot P_{\alpha+1}(x)   \\
		=&\sum_{i=1}^{n}\partial_{i}x_{i}P_{\alpha}(x) \\
		=&(n+\sum_{i=1}^{n}x_{i}\partial_{i})P_{\alpha}(x) \\
		=&(n+d\alpha)P_{\alpha}(x),
	\end{flalign*}
	which is equivalent to 
	\begin{flalign*}
		 f(\partial) \cdot P_{\alpha+1}(x)=\frac{n+d\alpha}{d}P_{\alpha}(x).
	\end{flalign*}
	\end{proof}
	\begin{theorem}\label{Tom_product}
	Suppose $f \in \mathbb{C}[x_{1},...,x_{n}], g \in \mathbb{C}[y_{1},...,y_{m}]$ are homogeneous polynomials with $\mathrm{deg}(f)=d$ and $\mathrm{deg}(g)=e$ satisfying $\frac{n}{d}=\frac{m}{e}$. If $f$ and $g$ each have the polynomial PDE property as a single polynomial, then $fg$ also has the polynomial PDE property as a single polynomial.
	\end{theorem}
	\begin{proof}
	Assume $P_{\alpha}(x)(\alpha \in \mathbb{Z}_{\ge 0})$ and $Q_{\alpha}(y)(\alpha \in \mathbb{Z}_{\ge 0})$ are the required polynomials for $f$ and $g$ respectively. We take $R_{\alpha}(x,y)=c_{\alpha}P_{\alpha}(x)Q_{\alpha}(y)$, where $c_{\alpha} \in \mathbb{C}$ is some non-zero constant depending on $\alpha$.
	By Lemma \ref{f_partial}, one can calculate as follows.
	\begin{flalign*}
		&\partial_{x_{i}}f(\partial_{x})g(\partial_{y})R_{\alpha+1} \\
		=&c_{\alpha+1}\partial_{x_{i}}f(\partial_{x})g(\partial_{y})P_{\alpha+1}(x)Q_{\alpha+1}(y)  \\
		=&c_{\alpha+1}x_{i}P_{\alpha}(x)(\alpha+\frac{m}{e})Q_{\alpha}(y) \\
		=&\frac{c_{\alpha+1}}{c_{\alpha}}(\alpha+\frac{m}{e})x_{i}R_{\alpha}.
	\end{flalign*}
	We take $c_{\alpha}$ such that $c_{\alpha}=c_{\alpha+1} \cdot (\alpha+\frac{m}{e})$, then $\partial_{x_{i}}f(\partial_{x})g(\partial_{y})R_{\alpha+1}=x_{i}R_{\alpha}$. Similarly, since $\frac{m}{e}=\frac{n}{d}$, we have $\partial_{y_{j}}g(\partial_{y})f(\partial_{x})R_{\alpha+1}=y_{j}R_{\alpha}$, giving the required statement. 
	\end{proof}
	
	\begin{theorem}\label{Tom_sum}
	Suppose $f \in \mathbb{C}[x_{1},...,x_{n}],g \in \mathbb{C}[y_{1},...,y_{m}]$ are homogeneous polynomials of degree $d(d \ge 2)$. If $f$ and $g$ have the polynomial PDE property as a single polynomial, then $f+g$ also has the polynomial PDE property as a single polynomial in $\mathbb{C}[x_{1},...,x_{n},y_{1},...,y_{m}]$.
	\end{theorem}
	\begin{proof}
     Let $P_{\alpha}(x),Q_{\alpha}(y)(\alpha \in \mathbb{Z}_{\ge 0})$ be the required polynomials for $f$ and $g$ respectively. We take $R_{\alpha}(x,y)=\sum_{k=0}^{\alpha}P_{k}(x)Q_{\alpha-k}(y)$. Then for any $x_{i}$ we have
	\begin{flalign*}
		&\partial_{x_{i}}(f(\partial_{x})+g(\partial_{y})) \cdot R_{\alpha+1}(x,y) \\
		=&\partial_{x_{i}}(f(\partial_{x})+g(\partial_{y})) \cdot \sum_{k=0}^{\alpha+1}P_{k}(x)Q_{\alpha+1-k}(y)  \\
		=&\partial_{x_{i}}f(\partial_{x}) \cdot \sum_{k=0}^{\alpha+1}P_{k}(x)Q_{\alpha+1-k}(y). 
	\end{flalign*}
	Because $d \ge 2$, $\partial_{x_{i}}f$ has degree at least $1$, so $\partial_{x_{i}}f(\partial_{x}) \cdot P_{0}(x)=0$. This implies that
	\begin{flalign*}
			&\partial_{x_{i}}(f(\partial_{x})+g(\partial_{y})) \cdot R_{\alpha+1}(x,y) \\
				=&\sum_{k=1}^{\alpha+1}x_{i}P_{k-1}(x)Q_{\alpha+1-k}(y) \\
			=&x_{i}R_{\alpha}(x,y).
	\end{flalign*}
	Similarly, for any $y_{j}$, $\partial_{y_{j}}(f(\partial_{x})+g(\partial_{y})) \cdot R_{\alpha+1}(x,y)=y_{j}R_{\alpha}(x,y)$. Thus one can verify that these $R_{\alpha}$ are the required polynomials for $f+g$.
	\end{proof}
		
		\section{Acknowledgement}

		Y. Chen and H. Zuo are supported by BJNSF Grant 1252009. H. Zuo is supported by NSFC Grants 12271280 and 12671056.

	\end{document}